\documentclass{amsproc}
\usepackage[all]{xy}
\usepackage{amssymb}
\usepackage{paralist}
\usepackage{mathdots}
\usepackage{geometry}

\usepackage[OT2,OT1]{fontenc}
\newcommand\cyr{%
\renewcommand\rmdefault{wncyr}%
\renewcommand\sfdefault{wncyss}%
\renewcommand\encodingdefault{OT2}%
\normalfont
\selectfont}
\DeclareTextFontCommand{\textcyr}{\cyr} 

\usepackage{xcolor}
\definecolor{myurlcolor}{rgb}{0,0,0.4}
\definecolor{mycitecolor}{rgb}{0,0.5,0}
\definecolor{myrefcolor}{rgb}{0.5,0,0}
\usepackage{hyperref}
\hypersetup{colorlinks,
linkcolor=myrefcolor,
citecolor=mycitecolor,
urlcolor=myurlcolor}

\input xy
\xyoption{all}
\CompileMatrices

\xymatrixcolsep{34pt}
\xymatrixrowsep{34pt}

\newcommand{\beq}{\begin{displaymath}}
\newcommand{\eeq}{\end{displaymath}}

\newcommand{\Z}{\mathbb{Z}}

\newcommand{\ra}{\rightarrow}

\newcommand{\lra}{\longrightarrow}

\newcommand{\A}{\mathcal{A}}
\newcommand{\C}{\mathcal{C}}
\newcommand{\D}{\mathcal{D}}
\newcommand{\W}{\mathcal{W}}
\newcommand{\Iso}{\mathrm{Iso}}
\newcommand{\Mor}{\mathrm{Mor}}
\newcommand{\id}{\mathrm{id}}
\newcommand{\Loc}{\mathrm{Loc}}
\newcommand{\coker}{\mathrm{coker}}
\newcommand{\im}{\mathrm{im}}
\newcommand{\Ab}{\mathbf{Ab}}
\newcommand{\K}{\mathcal{K}}
\newcommand{\Q}{\mathcal{Q}}

\swapnumbers
\newtheorem{prop}{Proposition}[section]
\newtheorem{thm}[prop]{Theorem}
\newtheorem*{thm*}{Theorem}
\newtheorem{defn}[prop]{Definition}
\newtheorem{cor}[prop]{Corollary}
\newtheorem{lem}[prop]{Lemma}
\theoremstyle{definition} %%% allows plain font for ex, rem and notn
\newtheorem{ex}[prop]{Example}
\newtheorem{rem}[prop]{Remark}

\renewcommand{\labelenumi}{(\alph{enumi})}
\renewcommand{\theenumi}{(\alph{enumi})}

\usepackage[draft]{fixme}

\title{Notes on Triangulated Categories}

\address{Tobias Fritz, Perimeter Institute for Theoretical Physics\\
Waterloo ON, Canada}
\email{tfritz@perimeterinstitute.ca}



\begin{document}

\maketitle
\begin{abstract}
We give an elementary introduction to the theory of triangulated categories covering their axioms, homological algebra in triangulated categories, triangulated subcategories, and Verdier localization. We try to use a minimal set of axioms for triangulated categories and derive all other statements from these, including the existence of biproducts. We conclude with a list of examples.
\end{abstract}
\tableofcontents

\section{Introduction}
\label{examples}

Triangulated categories are a convenient framework for homology and cohomology---not just in topology, but in all branches of mathematics that use functors having long exact sequences. In fact, any (co)homology theory has or at least should have a triangulated category as its underlying structure, which then can be analyzed by the long exact sequences. The triangulation is the categorical structure which keeps track of the theory's exactness properties. Consequently, triangulated categories arise in many fields of mathematics; see Section~\ref{se:ex}.

These notes constitute an elementary exposition of the basic theory of triangulated categories, starting with the definition of triangulated category and a detailed discussion of the axioms. The later sections provide details on homological algebra in triangulated categories and an overview of Verdier localization. Although other expositions exist---such as~\cite{Nee} or~\cite[Ch.~IV]{GM} just to name two---they frequently use a redundant set of axioms, as was observed by May~\cite{May,MayAdd}. We have tried to keep the number of axioms in the definition at a minimum, and from these axioms we then derive consequences which are included as additional axioms in other expositions, such as the existence of biproducts (Corollary~\ref{biproductsexist}).

These notes are intended for a readership with some basic understanding of category theory, and additive categories in particular, up to familiarity with biproducts writing morphisms between biproducts as matrices of morphisms between the individual summand objects. Since we develop the theory abstractly and only sketch some examples afterwards in Section~\ref{se:ex}, it will be helpful to already have some motivation for studying triangulated categories when reading these notes.

\subsection*{Notation and Terminology.} In all commutative diagrams in which the names of the objects are irrelevant, these objects are simply denoted by fat dots ``$\bullet$''. Except in cases where a commutativity statement is explicitly made, all diagrams commute. However, some diagrams have anticommmutative squares indicated by ``$\circleddash$''. Identity morphisms are pictured as double lines, ``\!\!$\xymatrix{\ar@{=}[r]&}$\!\!''. 

\subsection*{Acknowledgements.} In some parts, these notes are an elaboration on May's exposition~\cite[Sec.~2]{MayAdd}; see also~\cite{May}. Other ideas originate from Neeman's book~\cite{Nee}, while a few may be my own. This is a revised version of the second chapter of my MSc thesis (Diplomarbeit) written at the University of M\"unster in 2007. I would like to thank Joachim Cuntz for guidance and Peter May for some crucial comments on this revised version. The writing of this revised version was made possible through the support of a grant from the John Templeton Foundation.

Last but not least, I would like to thank all contributors to the \href{http://ncatlab.org/nlab/show/HomePage}{nLab} for having created an invaluable resource on everything categorical. Learning (higher) category theory will never be the same again.

\section{Definition}
\label{definition}

The basic data of a triangulated category are a preadditive category $\C$ with zero object $0\in\C$ and an additive automorphism $\Sigma:\C\lra\C$, whose inverse we write as $\Sigma^{-1}$. In other words, $\C$ is a category whose hom-sets are abelian groups such that composition is bilinear, and $\Sigma:\C\lra\C$ is a functor having an inverse $\Sigma^{-1} : \C \lra \C$; we will derive the existence of biproducts in Corollary~\ref{biproductsexist}, so that $\C$ is actually necessarily additive. Given an object $A\in\C$, call $\Sigma A$ the \emph{suspension} of $A$, and $\Sigma^{-1}A$ the \emph{desuspension} of $A$. 

We consider triples of composable morphisms of the form
\beq
\xymatrix{X\ar[r]^f & Y\ar[r]^g & Z\ar[r]^h & {}\Sigma X}
\eeq
Note that the last object in the sequence has to be \textit{equal} to the suspension of the first object (isomorphism is not sufficient). Such triples will be called \emph{candidate triangles}. Often, a candidate triangle will be denoted simply by its triple of morphisms $(f,g,h)$, omitting explicit mention of its objects. A morphism of candidate triangles is a diagram
\beq
\xymatrix{X_1\ar[r]^{f_1}\ar[d] & Y_1\ar[r]^{g_1}\ar[d] & Z_1\ar[r]^{h_1}\ar[d] & {}\Sigma X_1\ar[d]\\
X_2\ar[r]^{f_2} & Y_2\ar[r]^{g_2} & Z_2\ar[r]^{h_2} & {}\Sigma X_2}
\eeq
where the last vertical arrow has to be the suspension of the first one. This defines a category of candidate triangles in $\C$ in the obvious way. In particular, a morphism of candidate triangles is an isomorphism if and only if all the vertical maps are isomorphisms.

\begin{defn}
A \emph{triangulated category} is a preadditive category $\C$ with a zero object $0\in\C$ and equipped with an additive automorphism $\Sigma:\C\lra\C$ and a distinguished class of candidate triangles (the \emph{triangulation}),
\beq
\xymatrix{X\ar[r]^f & Y\ar[r]^g & Z\ar[r]^h & {}\Sigma X}
\eeq
succinctly called \emph{triangles}, subject to the following axioms:
\end{defn}

\renewcommand{\labelenumi}{(T\arabic{enumi})}
\renewcommand{\theenumi}{(T\arabic{enumi})}

\begin{compactenum}
\item\label{T1} The sequence
\beq
\xymatrix{X\ar@{=}[r] & X \ar[r] & 0 \ar[r] & {}\Sigma X}
\eeq
is a triangle for every object $X\in\C$.
\item\label{T2} Given any morphism $\xymatrix{X\ar[r]^f & Y}$, it can be completed to a triangle
\beq
\xymatrix{X\ar[r]^f & Y\ar[r]^g & Z\ar[r]^h & {}\Sigma X}
\eeq
\item\label{T3} Given isomorphic candidate triangles, if one of them is a triangle, then so is the other.

\item\label{T4} If $(f,g,h)$ is a triangle, then so is the shifted triple $(g,h,-\Sigma f)$.
\item\label{T5} The \emph{composition axiom}. Given all morphisms except the dashed ones in a diagram like
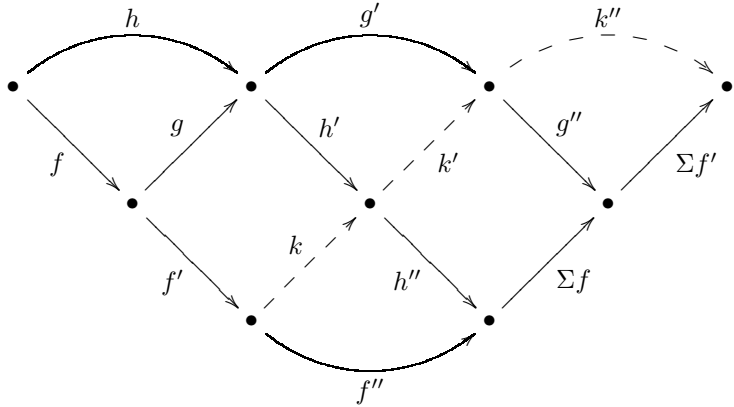
\beq
\xymatrix{{\bullet}\ar@(ru,lu)[rr]^h\ar[rd]_f & & {\bullet}\ar@(ru,lu)[rr]^{g'}\ar[rd]^{h'} & & {\bullet}\ar@(ru,lu)@{-->}[rr]^{k''}\ar[rd]^{g''} & & {\bullet}\\
& {\bullet}\ar[ur]^g\ar[rd]_{f'} & & {\bullet}\ar@{-->}[ru]_{k'}\ar[rd]_{h''} & & {\bullet}\ar[ru]_{\Sigma f'} &\\
& & {\bullet}\ar@(rd,ld)[rr]_{f''}\ar@{-->}[ru]^{k} & & {\bullet}\ar[ru]_{\Sigma f} & &}
\eeq
where only the upper left triangle is assumed to commute, and the straight paths $(f,f',f'')$, $(g,g',g'')$ and $(h,h',h'')$ are assumed to be triangles, then there exist $k$, $k'$ and $k''$ such that the whole diagram commutes and $(k,k',k'')$ is a triangle as well.\footnote{Of course, $k''$ is uniquely determined by commutativity.}
\end{compactenum}\phantom{lineends\\}

\noindent Now for some detailed comments on the axioms. First, here is an alternative drawing of the composition axiom diagram, which might give a better picture of the triangles involved:
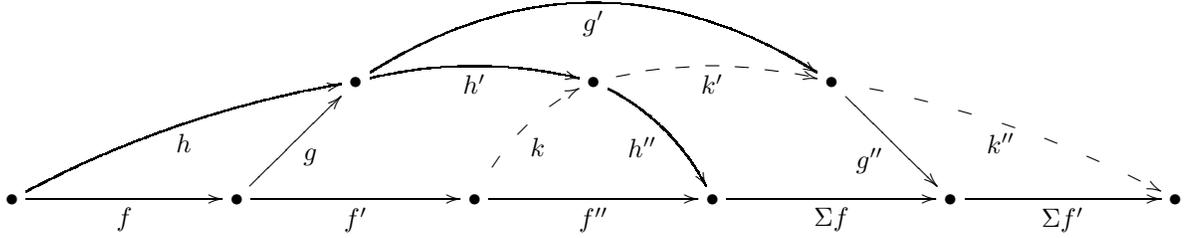
\beq
\xymatrix{&&&{\bullet}\ar@/^.5pc/[rr]_{h'}\ar@/^2.5pc/[rrrr]_{g'}&&{\bullet}\ar@/^.5pc/@{-->}[rr]_{k'}\ar@/^.5pc/[rd]_{h''}&&{\bullet}\ar@/^.5pc/@{-->}[rrrd]_{k''}\ar[rd]_{g''}&&&\\
{\bullet}\ar[rr]_f\ar@/^.5pc/[rrru]_h&&{\bullet}\ar[rr]_{f'}\ar[ru]_g&&{\bullet}\ar[rr]_{f''}\ar@/^.5pc/@{-->}[ru]_k&&{\bullet}\ar[rr]_{\Sigma f}&&{\bullet}\ar[rr]_{\Sigma f'}&&{\bullet}}
\eeq
The name of the axiom is due to viewing it as a diagram generated by the \emph{composition} $h=gf$.

\renewcommand{\labelenumi}{(\alph{enumi})}
\renewcommand{\theenumi}{(\alph{enumi})}

\begin{rem}
\label{defrems}
\begin{compactenum}
\item In a diagram where the objects are denoted by dots ``${\bullet}$'', it is tacitly assumed that for any (candidate) triangle, the last object is the suspension of the first.
\item In view of the upcoming Proposition~\ref{vanishcomp}, it is not surprising that the composition axiom does not make any commutativity premise besides $h=gf$, since commutativity of the rest of the diagram without the dashed arrows then is trivially implied.
\item In fact, the composition axiom is the only axiom which bounds the size of the triangulation from above, by postulating commutativity without assuming it as a premise; all other axioms state that something \emph{is} a triangle, while the composition axiom also gives an upper bound on the class of those candidate triangles which can be triangles. Without the composition axiom, taking \emph{all} candidate triangles to be triangles would yield a viable triangulation.
\item
\label{negativesexist}
It is not actually necessary to assume the existence of negatives (additive inverses) for morphisms in $\C$. It is sufficient to postulate only the structure of an abelian monoid on every morphism set $\C(X,Y)$, such that composition is bilinear and $\Sigma$ preserves addition. Then axiom~\ref{T4} can be replaced by:
\\
\begin{compactenum}
\item[(T4')]\label{T4'} If $(f,g,h)$ is a triangle, then there is a triangle $(g,h,\overline{f})$ such that $\overline{f}+\Sigma f=0$.
\end{compactenum}
\phantom{l\\}
Since $\Sigma$ is an automorphism, $\Sigma^{-1}\overline{f}$ is an additive inverse for $f$. Then by~\ref{T2}, any $f$ has an additive inverse, so that $\C$ is automatically preadditive. \ref{T4} now follows from~\ref{T4'} and the uniqueness of additive inverses.
\item\label{suspadd} Since biproducts can be defined in terms of a system of equations between morphisms~\cite[Sec.~VIII.2]{Mac}, any additive functor between preadditive categories preserves them. In particular, the suspension functor $\Sigma$ automatically preserves all biproducts that may exist in $\C$, meaning that there is a natural isomorphism $\Sigma(X\oplus Y)\cong \Sigma X \oplus \Sigma Y$ whenever $X\oplus Y$ exists.
\item\label{dualrem} Given a triangulated category $\C$, its dual $\C^{\mathrm{op}}$ is also canonically triangulated if we employ $\Sigma^{-1}$ as the suspension automorphism of $\C^{\mathrm{op}}$. Thus, all statements about triangulated categories have corresponding dual statements which are then automatically true. This is not obvious, since the axioms are not formulated in a self-dual manner, but see Remark~\ref{dual}.
\item Many statements on homological algebra in triangulated categories are analogues of well-known statements from the theory of abelian categories, if one regards triangles as the counterparts of short exact sequences. For example, given subobject inclusions
\beq
\xymatrix{A\:\ar@{^{(}->}[r] & B\:\ar@{^{(}->}[r] & C}
\eeq
in an abelian category, the isomorphism theorem $(C/A)/(B/A)\cong (C/B)$ is a close analogue of the composition axiom~\ref{T5}. This can be seen best by drawing the isomorphism theorem as a braid diagram of short exact sequences:
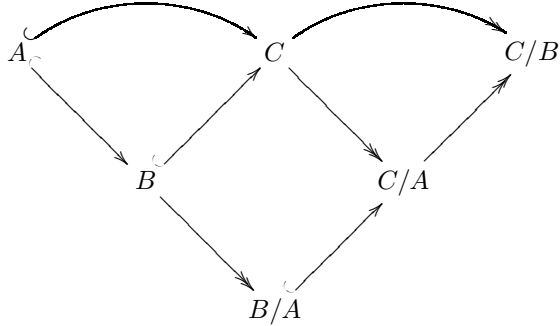
\beq
\xymatrix@!@-12pt{\\A\ar@(ru,lu)@{^{(}->}[rr]\ar@{^{(}->}[rd] && C\ar@{->>}[dr]\ar@(ru,lu)@{->>}[rr] && C/B\\
& B\ar@{->>}[rd]\ar@{^{(}->}[ru] && C/A\ar@{->>}[ur]\\
&& B/A\ar@{^{(}->}[ru]}
\eeq
Similarly, the yet-to-follow statements Remark~\ref{weakkernel} and Corollary~\ref{iso} correspond to the proposition that a morphism in an abelian category is an isomorphism if and only if its kernel and cokernel both vanish. Moreover, Lemma~\ref{3x3lemma} is an analogue of the nine lemma~\cite[13.5.6]{Sch1}.
\item Although widely used in many areas of mathematics, the notion of triangulated category is often not considered satisfactory. One problem is that taking mapping cones like in Proposition~\ref{mappingcone} is not a functorial operation as indicated by the non-uniqueness of filling morphisms (Remark~\ref{nonuniquefm}); another problem is that the formation of homotopy limits and homotopy colimits~\cite{Dugger,Riehl} cannot always be expressed in terms of data in the triangulated category alone. See Section~\ref{refine}.
\item There is an interesting higher-dimensional variant of~\ref{T5} which has been investigated by Schmidt~\cite{Schm}; Proposition~\ref{triplecomp} proves that part of it follows from the other axioms.
\item
Finally, there are numerous other variations both on the set of axioms and on the terminology: 
\begin{compactenum}
\item Many authors use the term ``distinguished triangle'' or ``exact triangle'' instead of the simpler ``triangle'' when referring to an element of the triangulation. Also, the literature knows several notions of ``candidate triangle'' and ``pre-triangle'' differing from the one given above.
\item Often $\Sigma$ is not required to be an automorphism, but merely an autoequivalence. This can lead to trouble when applying the desuspension functor to triangles, because 
\beq
\xymatrix{\Sigma^{-1}X\ar[r] & {\bullet}\ar[r] & {\bullet}\ar[r] & X}
\eeq
is not even a candidate triangle if $\Sigma(\Sigma^{-1}X)\neq X$. Some authors who take $\Sigma$ to be an autoequivalence seem to ignore this issue, however. Since equivalent categories are identical as far as the theory of the mathematical structures represented by them is concerned, this is not a problem and can be fixed by the following well-known construction: define a new category $\C'$ with objects the class of pairs 
\beq
(X,n)\quad\textrm{with}\quad X\in\C,\:n\in\Z
\eeq
and morphism sets
\beq
\C'((X,n),(Y,n)):=\C(\Sigma^n X,\Sigma^m Y)
\eeq
where composition of morphisms also is inherited from $\C$.
One can think of $(X,n)$ as being an $n$-fold formal suspension of $X$. Then the new suspension functor
\[
\Sigma'\: :\: \C'\lra\C',\qquad (X,n)\longmapsto (X,n+1)
\]
is an automorphism such that the pair $(\C',\Sigma')$ is equivalent to the pair $(\C,\Sigma)$.
\item Especially in the context of derived categories (Example~\ref{derived}), it is useful to replace $\C$ by an associated $\Z$-graded category $\C_\textrm{graded}$, where a morphism from $X$ to $Y$ of degree $n$ is an element of $\C(X,\Sigma^nY)$, so that
\beq
\C_\textrm{graded}(X,Y)=\bigoplus_{n\in\Z}\C(X,\Sigma^n Y).
\eeq
This emphasizes the interpretation of $\Sigma$ as shifting the dimension of an object by $1$. Candidate triangles then have the form
\beq
\xymatrix{X\ar[rr]^{\textrm{degree }0} && Y\ar[ld]^{\textrm{degree }0}\\
& Z\ar[lu]^{\textrm{degree }1}}
\eeq
This explains the terminology ``triangle''. Using this graded point of view is also a difference in notation and exposition; the theory is the same.
\item When using this graded notation, one can write the composition axiom in the form of an octahedron. This is why it is often known as the \emph{octahedral axiom}. A disadvantage of the octahedral shape is that the octahedron diagram cannot be drawn in the plane without self-intersections.
\item There are also minor variations on the axioms. For example, the upcoming Proposition~\ref{fillingmorphism} is usually taken to be an axiom, although it follows easily from the other axioms. Neeman~\cite{Nee} dispenses with the composition axiom or octahedral axiom completely and uses as his version that filling morphisms can be chosen such that they are ``good''. 
\item\label{biprodrem} Similarly, the category $\C$ is usually assumed to be additive. In contrast to this, we derive the existence of biproducts in $\C$ in Corollary~\ref{biproductsexist}.
\end{compactenum}
\end{compactenum}
\end{rem}

\section{Puppe sequences and homological functors}
\label{homalgebra}

We now study triangles in $\C$ and how they give rise to long exact sequences. This is the beginning of \emph{homological algebra} in a triangulated category.

\begin{prop}
\label{signs}
If $(f,g,h)$ is a triangle, then so are $(f,-g,-h)$, $(-f,g,-h)$ and $(-f,-g,h)$. In other words, we can insert two signs wherever we want.
\end{prop}

\begin{proof}
This immediately follows from~\ref{T3}, because
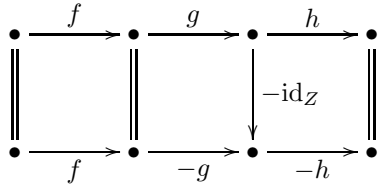
\beq
\xymatrix{ \bullet \ar[r]^f\ar@{=}[d] & \bullet\ar[r]^g\ar@{=}[d] & \bullet\ar[r]^h\ar[d]^{-\id_Z} & \bullet \ar@{=}[d]\\
\bullet \ar[r]_f & \bullet\ar[r]_{-g} & \bullet\ar[r]_{-h} & \bullet }
\eeq
is an isomorphism of candidate triangles, so that the second row $(f,-g,-h)$ is a triangle as well. The other two cases work similarly.
\end{proof}

In contrast, changing the sign of only one of the morphisms in a triangle does in general not yield a second triangle. However, we can define a new triangulation on $\C$ by taking $(f,g,h)$ as a triangle in a new triangulation if and only if its negative $(-f,-g,-h)$ is a triangle in the ``old'' triangulation. For more investigations on how unique the triangulation is when the category is given together with the suspension functor, see~\cite[Sec.~2.4]{May} or~\cite{Bal}.

The main features of triangulated categories are homological functors and long exact sequences; capturing these phenomena is what triangulated categories are for. Intuitively speaking, given a ``half-exact'' functor, we can construct its ``$n$-th derived functor'' by pre-composing it with $\Sigma^n$. We will soon be able to make this precise. 

By repeated application of~\ref{T4}, we can continue any triangle
\beq
\xymatrix{X\ar[r]^f & Y\ar[r]^g & Z\ar[r]^h & {}\Sigma X}
\eeq
to a sequence
\beq
\xymatrix{X\ar[r]^f & Y\ar[r]^g & Z\ar[r]^h & {}\Sigma X\ar[r]^{\Sigma f} & {}\Sigma Y\ar[r]^{\Sigma g} & {}\Sigma Z\ar[r]^{\Sigma h} & {}\Sigma^2 X\ar[r]^{\Sigma^2 f} & {}\ldots}
\eeq
which consists of an alternating sequence of triangles, like $(f,g,h)$, and negatives of triangles, like $(g,h,\Sigma f)$; that $(g,h,\Sigma f)$ is the negative of a triangle follows from the previous proposition together with~\ref{T4}. In algebraic topology, this infinite sequence of morphisms is known as the \emph{Puppe sequence}, and this is also the terminology that we employ. Using Puppe sequences, one can extend the composition axiom diagram to the braid diagram
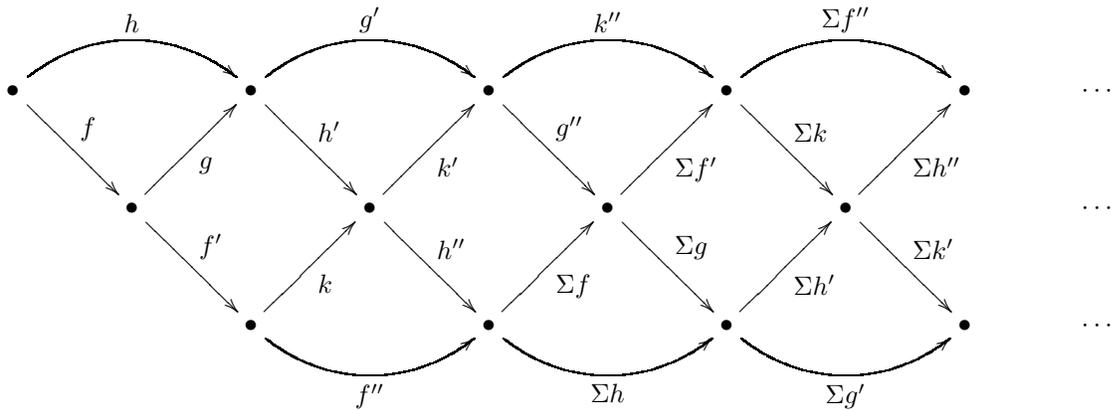
\beq
\xymatrix{{\bullet}\ar@(ru,lu)[rr]^h\ar[rd]^f & & {\bullet}\ar@(ru,lu)[rr]^{g'}\ar[rd]^{h'} & & {\bullet}\ar@(ru,lu)[rr]^{k''}\ar[rd]^{g''} & & {\bullet}\ar@(ru,lu)[rr]^{\Sigma f''}\ar[rd]^{\Sigma k} & & {\bullet} & {}\cdots\\
& {\bullet}\ar[ur]_g\ar[rd]^{f'} & & {\bullet}\ar[ru]_{k'}\ar[rd]^{h''} & & {\bullet}\ar[ru]_{\Sigma f'}\ar[rd]^{\Sigma g} & & {\bullet}\ar[rd]^{\Sigma k'}\ar[ru]_{\Sigma h''} & &  {}\cdots\\
& & {\bullet}\ar@(rd,ld)[rr]_{f''}\ar[ru]_{k} & & {\bullet}\ar[ru]_{\Sigma f}\ar@(rd,ld)[rr]_{\Sigma h} & & {\bullet}\ar@(rd,ld)[rr]_{\Sigma g'}\ar[ru]_{\Sigma h'} & & {\bullet} & {}\cdots}
\eeq
where now Puppe sequences go along the strands. Commutativity simply follows from commutativity of the composition axiom diagram together with functoriality of $\Sigma$. 

We will see that the Puppe sequence as well as the braid diagram can actually both also be continued to the left by using $\Sigma^{-1}$. However, this requires a bit more machinery to prove.

\begin{prop}
\label{fillingmorphism}
Given a diagram
\beq
\xymatrix{X\ar[r]^f\ar[d]^j & Y\ar[r]^g\ar[d]^k & Z\ar[r]^h & {}\Sigma X\\
X'\ar[r]^{f'} & Y'\ar[r]^{g'} & Z'\ar[r]^{h'} & \Sigma X'}
\eeq
where the square is assumed commutative and the rows are triangles, there exists a \emph{filling morphism} $m:Z\lra Z'$ completing it to a morphism of triangles:
\beq
\xymatrix{X\ar[r]^f\ar[d]^j & Y\ar[r]^g\ar[d]^k & Z\ar[r]^h\ar@{-->}[d]^m & {}\Sigma X\ar[d]^{\Sigma j}\\
X'\ar[r]^{f'} & Y'\ar[r]^{g'} & Z'\ar[r]^{h'} & {}\Sigma X'}
\eeq
\end{prop}

Filling morphisms are among the main properties of triangulated categories. This might be why this proposition is often taken as part of the axiom system, although this is redundant since it follows from the other axioms. In the next subsection, we will see that filling morphisms are the crucial ingredient for the exactness properties of triangles.

\begin{proof}
Let us consider two easy cases first, namely those where one of the given vertical morphisms is an identity morphism. Start with $j=\id_X$. Then the alleged filling morphism $m_1:Z\lra Z'$ in
\beq
\xymatrix{X\ar[r]^f\ar@{=}[d] & Y\ar[r]^g\ar[d]^k & Z\ar[r]^h\ar@{-->}[d]^{m_1} & {}\Sigma X\ar@{=}[d]^{\Sigma j}\\
X\ar[r]^{f'} & Y'\ar[r]^{g'} & Z'\ar[r]^{h'} & {}\Sigma X}
\eeq
follows from an application of the composition axiom to the equation $f'=kf$ together with a third triangle obtained from $k$ via~\ref{T2}; recall that the composition axiom does not make any additional commutativity premise. In the same way, the filling morphism $m_2:Z\lra Z'$ with $k=\id_Y$ in the diagram
\beq
\xymatrix{X\ar[r]^f\ar[d]^j & Y\ar[r]^g\ar@{=}[d] & Z\ar[r]^h\ar@{-->}[d]^{m_2} & {}\Sigma X\ar[d]^{\Sigma j}\\
X'\ar[r]^{f'} & Y\ar[r]^{g'} & Z'\ar[r]^{h'} & {}\Sigma X'}
\eeq
also follows from the composition axiom applied to the equation $f=f'j$. Now the general case can be factored into two of these easy cases:
\beq
\xymatrix{X\ar[r]^f\ar@{=}[d] & Y\ar[r]^g\ar[d]^k & Z\ar[r]^h\ar@{-->}[d]^{m_1} & {}\Sigma X\ar@{=}[d]\\
X\ar[r]^{kf=f'j}\ar[d]^j & Y'\ar[r]\ar@{=}[d] & W\ar[r]\ar@{-->}[d]^{m_2} & {}\Sigma X\ar[d]^{\Sigma j}\\
X'\ar[r]^{f'} & Y'\ar[r]^{g'} & Z'\ar[r]^{h'} & {}\Sigma X'}
\eeq
where the triangle in the middle row was also obtained through an application of~\ref{T2}.
\end{proof}

However, the filling morphism $m:Z\lra Z'$ is in general not unique. Remark~\ref{nonuniquefm} will give a very general example of this phenomenon.

\begin{prop}
\label{vanishcomp}
In a triangle $(f,g,h)$, we have $gf=0$, $hg=0$ and $(\Sigma f)\circ h=0$.
\end{prop}

\begin{proof}
By~\ref{T1},~\ref{T4} and Proposition~\ref{signs}, $\xymatrix{X\ar[r] & 0\ar[r] & {}\Sigma X\ar@{=}[r] & {}\Sigma X}$ is a triangle. Since the premises of Proposition~\ref{fillingmorphism} are trivially verified, there is a filling morphism
\beq
\xymatrix{X\ar[r]^f\ar@{=}[d] & Y\ar[r]^g\ar[d] & Z\ar[r]^h\ar@{-->}[d]^m & {}\Sigma X\ar@{=}[d]\\
X\ar[r] & 0\ar[r] & \Sigma X\ar@{=}[r] & {}\Sigma X}
\eeq
However, by commutativity of the right square, we need to have $m=h$, and commutativity of the middle square then yields $hg=0$. In other words, for every triangle the composition of the second and third morphism vanishes. In particular, this is also the case for the triangles $(g,h,-\Sigma f)$ and $(h,-\Sigma f,-\Sigma g)$, so that $(\Sigma f)\circ h=0$ and $\Sigma(gf)=0$. But since $\Sigma$ is a faithful functor, we also have $gf=0$.
\end{proof}

Now is the time that we can turn to homology and see the powerful machinery of triangles in action.

\begin{defn}
An additive functor $H:\C\lra\Ab$ is called a \emph{homological functor} if $H$ is \emph{half-exact} in the sense that for every triangle
\beq
\xymatrix{X\ar[r]^f & Y\ar[r]^g & Z\ar[r]^h & {}\Sigma X}
\eeq
the sequence of abelian groups
\beq
\xymatrix{H(X)\ar[r]^{H(f)} & H(Y)\ar[r]^{H(g)} & H(Z)}
\eeq
is exact.
\end{defn}

Now since a Puppe sequence consists of alternating triangles and negatives of triangles, it follows that $H$ turns any Puppe sequence into a long exact sequence
\beq
\xymatrix{H(X)\ar[r]^{H(f)} & H(Y)\ar[r]^{H(g)} & H(Z)\ar[r]^{H(h)} & H(\Sigma X)\ar[r]^{H(\Sigma f)} & H(\Sigma Y)\ar[r]^{H(\Sigma g)} & {}\ldots}
\eeq
Thus for $n\geq 0$, we may regard $H\circ\Sigma^n$ as the $n$-th \emph{right derived} functor of $H$.

\begin{rem}
	If $H$ is contravariant, but otherwise the analogous condition holds, then $H$ is called a \emph{cohomological functor}. We will not really need this concept, but it is good to keep in mind that everything that follows applies to cohomological functors in the same way as it does to homological functors. Another point of view is that one can also define homological functors with values in an arbitrary abelian category possibly different from $\Ab$, and then a cohomological functor is simply a homological functor taking values in $\Ab^{\mathrm{op}}$.
\end{rem}

\begin{prop}
\label{homhom}
The $\mathrm{Hom}$-functor $\C(W,\cdot):\C\lra\Ab$ for any fixed object $W\in\C$ is homological.
\end{prop}

\begin{proof}
Let a triangle
\beq
\xymatrix{X\ar[r]^f & Y\ar[r]^g & Z\ar[r]^h & {}\Sigma X}
\eeq
be given. By the previous proposition, it is clear that $H(g)\circ H(f)=H(gf)=0$ for any additive functor $H$, and this holds for $\C(W,\cdot)$ in particular. To establish exactness of
\beq
\xymatrix{ {\C(W,X)}\ar[r]^{f\circ} & {\C}(W,Y)\ar[r]^{g\circ} & {\C}(W,Z)}
\eeq
it thus remains to be shown that, given some $v\in\C(W,Y)$ with $gv=0$, there is some $u\in\C(W,X)$ with $v=fu$. Since $\Sigma$ is an automorphism, it is sufficient to exhibit a morphism $k$ with $\Sigma v=(-\Sigma f)\circ k$, since then we can take $u:=-\Sigma^{-1}k$. Such a $k$ in turn can be obtained as a filling morphism in the diagram
\beq
\xymatrix{W\ar[r]\ar[d]^v & 0 \ar[r]\ar[d] & {}\Sigma W\ar@{=}[r]\ar@{-->}[d]^k & {}\Sigma W\ar[d]^{\Sigma v}\\
Y\ar[r]^g & Z\ar[r]^h & {}\Sigma X\ar[r]^{-\Sigma f} & {}\Sigma Y}
\eeq
which completes the proof.
\end{proof}

\begin{rem}
\label{brownrep}
If the triangulated category satisfies certain conditions which often hold in practice, it can actually be shown that a homological functor is representable, meaning that it is naturally isomorphic to one of the form $\C(W,\cdot)$ for some appropriate $W\in\C$, if and only if it preserves products. Dually, there are conditions which guarantee that a cohomological functor is of the form $\C(\cdot,W)$ if and only if it turns coproducts into products. This is a general form of \emph{Brown representability} familiar from algebraic topology. For more on this topic, see~\cite[Ch. 8]{Nee}.
\end{rem}

On several occasions, the Yoneda lemma will be very useful in our context. The following version is perfectly sufficient for our purposes:

\begin{lem}[{Yoneda lemma~\cite[III.2]{Mac}}]
\label{yoneda}
Suppose $f:X\ra Y$ is a morphism such that the natural transformation
\beq
\C(\cdot,f)\: :\: \C(\cdot,X)\lra\C(\cdot,Y),\qquad g\longmapsto fg
\eeq
of contravariant functors is an isomorphism. Then $f$ itself is an isomorphism.
\end{lem}

\begin{cor}
\label{yonedacor}
If $H(f)$ is an isomorphism for all homological functors $H$, then the morphism $f$ itself is an isomorphism.
\end{cor}

\begin{proof}
For any object $W\in\C$, we have the homological functor $\C(W,\cdot)$. So by the assumption, $\C(W,f)\: :\: \C(W,X)\lra\C(W,Y)$ is an isomorphism for all $W$. The Yoneda lemma applies to show that $f$ is an isomorphism itself.
\end{proof}

At times, it is useful to apply this result to detect that certain candidate triangles are in fact triangles:

\begin{cor}
Let
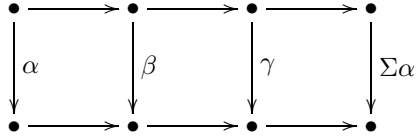
\[
\xymatrix{ \bullet \ar[r] \ar[d]^\alpha & \bullet \ar[r] \ar[d]^\beta & \bullet \ar[r] \ar[d]^\gamma & \bullet \ar[d]^{\Sigma\alpha} \\
	\bullet \ar[r] & \bullet \ar[r] & \bullet \ar[r] & \bullet }
\]
be a morphism of candidate triangles. If $H(\alpha)$, $H(\beta)$ and $H(\gamma)$ are isomorphisms under any homological functor $H$ and either the first or the second row is a triangle, then so is the other.
\end{cor}

The proof is immediate from Corollary~\ref{yonedacor} together with~\ref{T3}.

\section{Mapping cones and their weak functoriality}
\label{mappingconessection}

\begin{prop}
\label{mappingcone}
If in the diagram
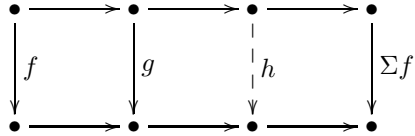
\beq
\xymatrix{{\bullet}\ar[r]\ar[d]^f & {\bullet}\ar[r]\ar[d]^g & {\bullet}\ar[r]\ar@{-->}[d]^h & {\bullet}\ar[d]^{\Sigma f}\\
{\bullet}\ar[r] & {\bullet}\ar[r] & {\bullet}\ar[r] & {\bullet}}
\eeq
where the rows are triangles, both $f$ and $g$ are isomorphisms, then any filling morphism $h$ is also an isomorphism.
\end{prop}

\begin{proof}
Extend the diagram to the first five terms of the Puppe sequence
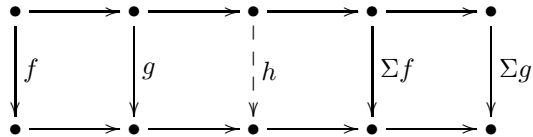
\beq
\xymatrix{{\bullet}\ar[r]\ar[d]^f & {\bullet}\ar[r]\ar[d]^g & {\bullet}\ar[r]\ar@{-->}[d]^h & {\bullet}\ar[d]^{\Sigma f}\ar[r] & {\bullet}\ar[d]^{\Sigma g}\\
{\bullet}\ar[r] & {\bullet}\ar[r] & {\bullet}\ar[r] & {\bullet}\ar[r] & {\bullet}}
\eeq
Applying any homological functor $H$ to this gives a chain map between five-term exact sequences. Then $H(h)$ is an isomorphism by the five lemma (e.g.~\cite[VIII.4.4]{Mac}). Thus $h$ is an isomorphism itself by Corollary~\ref{yonedacor}.
\end{proof}

\begin{cor}[Uniqueness of the mapping cone]
Each morphism $f:X\lra Y$ determines a unique isomorphism class of objects $Z$ which fit into a triangle
\beq
\xymatrix{X\ar[r]^f & Y\ar[r] & Z\ar[r] & {}\Sigma X}.
\eeq
\end{cor}

\begin{proof}
The existence of such a $Z$ is~\ref{T2}. It remains to show that any two such objects $Z$ and $Z'$ are isomorphic. Suppose $Z$ and $Z'$ are given with their respective triangles,
\beq
\xymatrix{X\ar[r]^f\ar@{=}[d] & Y\ar[r]\ar@{=}[d] & Z\ar[r] & {}\Sigma X\\
X\ar[r]^f & Y\ar[r] & Z'\ar[r] & {}\Sigma X}
\eeq
then Proposition~\ref{fillingmorphism} guarantees the existence of a filling morphism $j$
\beq
\xymatrix{X\ar[r]^f\ar@{=}[d] & Y\ar[r]\ar@{=}[d] & Z\ar[r]\ar@{-->}[d]^j & {}\Sigma X\ar@{=}[d]\\
X\ar[r]^f & Y\ar[r] & Z'\ar[r] & {}\Sigma X}
\eeq
which then must be an isomorphism by the previous proposition.
\end{proof}

Any object in this isomorphism class is referred to as a \emph{mapping cone} of $f$ (or often as \emph{the} mapping cone of $f$). Such a mapping cone is written as $C_f$ and fits into a \emph{mapping cone triangle}
\beq
\xymatrix{X\ar[r]^f & Y\ar[r]^{i_f} & C_f\ar[r]^{p_f} & {}\Sigma X}
\eeq
with a \emph{mapping cone inclusion} $i_f$ (not necessarily a monomorphism) and a \emph{mapping cone projection} $p_f$ (not necessarily an epimorphism); this terminology is intuitive when thinking about mapping cones for continuous maps between topological spaces.

Of course \emph{any} triangle is a mapping cone triangle. This point of view however has the advantage that we may think of a triangle as (almost) uniquely determined by its first morphism $f$. We may summarize this by noting that we can turn the mapping cone construction into something resembling a functor\footnote{The category of morphisms $\Mor(\C)$ has as objects the morphisms of $\C$, and as morphisms commutative squares in $\C$.} $\Mor(\C)\lra\C$: to every object $f\in\Mor(\C)$, associate a mapping cone $C_f$. To every morphism in $\Mor(\C)$, associate\footnote{Note that for large categories, showing the existence of such an assignment requires the axiom of global choice.} a filling morphism by Proposition~\ref{fillingmorphism}. Due to the non-uniqueness of filling morphisms (Remark~\ref{nonuniquefm}), we cannot expect this construction to preserve composition, and in fact under some hypotheses on $\C$ that often hold in practice, one can show that a functorial choice is not possible~\cite{Stev}. But although assigning a mapping cone to each morphism is generally not functorial, it at least maps isomorphisms to isomorphisms by Proposition~\ref{mappingcone}.

\begin{cor}
\label{mapconeiso}
Given a diagram like
\beq
\xymatrix{{}\bullet\ar[r]^f\ar[d]^g_\sim & {}\bullet\ar[d]^h_\sim\\
{}\bullet\ar[r]_{f'} & {}\bullet}
\eeq
then any mapping cone $C_f$ is isomorphic to any mapping cone $C_{f'}$.
\end{cor}

\begin{proof}
Completing $f$ and $f'$ to a triangle respectively containing $C_f$ and $C_{f'}$ allows a filling morphism $C_f\lra C_{f'}$, which then must be an isomorphism.
\end{proof}

Although this result might seem obvious in the general theory, it is actually far from trivial for relevant cases of $\C$, for example in stable homotopy theory (Example~\ref{spectra}).

Now what do those mapping cones actually mean? It may be familiar from algebraic topology that one can define homology not just for spaces, but also for continuous maps; this is a very general form of ``relative homology''. Then one also obtains a long exact sequence connecting the homology of a composition $gf$ with the homologies of $f$ and $g$ themselves---see for example~\cite{BottTu}. The interpretation is that the homology of a map is the homology of its mapping cone, where the mapping cone represents the map as an object in the triangulated category. Furthermore, the Puppe sequence of the triangle connecting the three mapping cones of a composition $gf$---obtained from~\ref{T5}---is the long exact sequence just mentioned! 

There is another consequence of the quasi-uniqueness of mapping cones which is a useful criterion for detecting when two triangulations on a given category coincide: it is sufficient that each triangle in one of them is also a triangle in the other. 

\begin{cor}
\label{triangscoincide}
If $\C$ is a preadditive category with suspension functor $\Sigma$ and triangulations $\mathcal{T}_1$ and $\mathcal{T}_2$ such that $\mathcal{T}_1\subseteq\mathcal{T}_2$, then $\mathcal{T}_1=\mathcal{T}_2$.
\end{cor}

\begin{proof}
We have to show that any triangle $(f,g,h)$ which lies in $\mathcal{T}_2$ also automatically lies in $\mathcal{T}_1$. By the axiom~\ref{T2} for $\mathcal{T}_2$, there is a triangle $(f,g',h')\in\mathcal{T}_1$. But then by Proposition~\ref{mappingcone} applied with respect to $\mathcal{T}_2$, the triples $(f,g,h)$ and $(f,g',h')$ are isomorphic. So $(f,g,h)\in\mathcal{T}_1$ is implied by~\ref{T3}.
\end{proof}

Finally, here is one more crucial consequence of Proposition~\ref{mappingcone}:

\begin{cor}
\label{iso}
A morphism $f$ is an isomorphism if and only if there is a triangle
\beq
\xymatrix{{\bullet}\ar[r]^f & {\bullet}\ar[r] & 0\ar[r] & {\bullet}}
\eeq
\end{cor}

\begin{proof}
If $f$ is an isomorphism, then it is isomorphic in $\Mor(\C)$ to an identity morphism, and therefore fits into such a triangle by~\ref{T1}, Proposition~\ref{mappingcone} and~\ref{T3}. Conversely, if the above is a triangle, then the Puppe sequence with respect to any homological functor $H$ shows that $H(\Sigma f)$ is necessarily an isomorphism, and then so is $f$ by Corollary~\ref{yonedacor} and the invertibility of $\Sigma$.
\end{proof}

We will make use of this result on various occasions.

\begin{prop}
\label{shift}
$(f,g,h)$ is a triangle if and only if the shifted triple $(g,h,-\Sigma f)$ is.
\end{prop}

\begin{proof}
The ``only if'' part is~\ref{T4}. For the ``if'' part, complete $f:X\lra Y$ to a triangle $(f,g',h')$. Then both $(-\Sigma f,-\Sigma g,-\Sigma h)$ and $(-\Sigma f,-\Sigma g',-\Sigma h')$ are triangles, which are isomorphic by Proposition~\ref{fillingmorphism} and Proposition~\ref{mappingcone}. Applying $\Sigma^{-1}$ proves the assertion by~\ref{T3}.
\end{proof}

Thus if $(f,g,h)$ is a triangle, then so is $(-\Sigma^{-1}h,f,g)$. This proposition allows us to define left derived functors by continuing the Puppe sequence to the left. Similarly, the braid diagram also can be extended to the left.

\begin{cor}
\begin{compactenum}
\item The Puppe sequence
\beq
\xymatrix{{\ldots}\ar[r]^{\Sigma^{-1}g} & {\Sigma^{-1}Z}\ar[r]^{\Sigma^{-1}h} & X\ar[r]^f & Y\ar[r]^g & Z\ar[r]^h & {\Sigma X}\ar[r]^{\Sigma f} & {\Sigma Y}\ar[r]^{\Sigma g} & {\ldots}}
\eeq
consists of triangles and negatives of triangles which alternate. It maps to a long exact sequence under any homological functor.
\item For any $h=gf$, the extended braid diagram
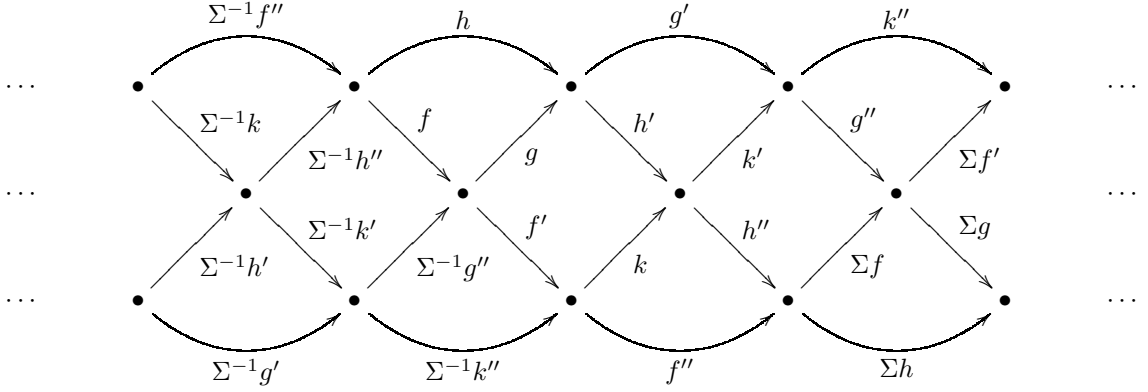
\beq
\xymatrix@-4pt{{\cdots} & {\bullet}\ar@(ru,lu)[rr]^{\Sigma^{-1}f''}\ar[rd]^{\Sigma^{-1}k} && {\bullet}\ar@(ru,lu)[rr]^h\ar[rd]^f & & {\bullet}\ar@(ru,lu)[rr]^{g'}\ar[rd]^{h'} & & {\bullet}\ar@(ru,lu)[rr]^{k''}\ar[rd]^{g''} & & {\bullet} & \cdots\\
\cdots && {\bullet}\ar[rd]^{\Sigma^{-1}k'}\ar[ru]_{\Sigma^{-1}h''} && {\bullet}\ar[ur]_g\ar[rd]^{f'} & & {\bullet}\ar[ru]_{k'}\ar[rd]^{h''} & & {\bullet}\ar[ru]_{\Sigma f'}\ar[rd]^{\Sigma g} & & {\cdots}\\
{\cdots} & {\bullet}\ar[ru]_{\Sigma^{-1}h'}\ar@(rd,ld)[rr]_{\Sigma^{-1}g'} && {\bullet}\ar@(rd,ld)[rr]_{\Sigma^{-1}k''}\ar[ru]_{\Sigma^{-1}g''} && {\bullet}\ar@(rd,ld)[rr]_{f''}\ar[ru]_{k} & & {\bullet}\ar[ru]_{\Sigma f}\ar@(rd,ld)[rr]_{\Sigma h} & & {\bullet} & {\cdots}}
\eeq
has Puppe sequences going along the strands and thus maps to a braid diagram of long exact sequences under a homological functor.
\end{compactenum}
\end{cor}

\begin{rem}
\label{shiftcone}
By three applications of Proposition~\ref{shift} in combination with Proposition~\ref{signs}, we can calculate the mapping cone of the suspension of a morphism as $C_{\Sigma f}\cong\Sigma C_f$.
\end{rem}

\begin{rem}
\label{dual}
It is now clear that the dual of a triangulated category is triangulated in a canonical way, as announced in Remark~\ref{defrems}\ref{dualrem}.
\end{rem}

One further interpretation of mapping cones is the following:

\begin{rem}
\label{weakkernel}
Consider any triangle
\beq
\xymatrix{\bullet \ar[r]^f & {\bullet}\ar[r]^{i_f} & C_f\ar[r]^{p_f} & \bullet }
\eeq
where we have named the morphisms and the third object in a suggestive manner. Then 
\beq
\xymatrix{{\Sigma^{-1}C_f}\ar[r]^(.6){\Sigma^{-1}p_f} & {\bullet}\ar[r]^f & {\bullet}\ar[r]^{i_f} & C_f}
\eeq
is the negative of a triangle, and:
\begin{compactenum}
\item The desuspended mapping cone projection $\Sigma^{-1}p_f$ is a weak kernel of $f$. In other words, $f\circ (\Sigma^{-1}p_f)=0$ and given any morphism $g$ such that $fg=0$, there exists a (non-unique) lift along $\Sigma^{-1}p_f$:
\beq
\xymatrix{{\bullet}\ar@{-->}[d]\ar[rd]^g\\
\Sigma^{-1}C_f\ar[r]_(.6){\Sigma^{-1}p_f} & {\bullet}\ar[r]_f & {\bullet}\ar[r]_{i_f} & C_f}
\eeq
\item The mapping cone inclusion $i_f$ is a weak cokernel of $f$. In other words, $i_ff=0$ and given any morphism $h$ such that $hf=0$, there exists a (non-unique) extension along $i_f$:
\beq
\xymatrix{{\Sigma^{-1}C_f}\ar[r]^(.6){\Sigma^{-1}p_f} & {\bullet}\ar[r]^f & {\bullet}\ar[r]^{i_f}\ar[rd]_h & C_f\ar@{-->}[d]\\
&&&{\bullet}}
\eeq
\end{compactenum}

\begin{proof}
By duality, it is sufficient to prove the second statement. But this follows immediately from the simple filling morphism diagram
\beq
\xymatrix{{\bullet}\ar[r]^f\ar[d]^0 & {\bullet}\ar[r]^{i_f}\ar[d]^h & C_f\ar[r]^{p_f}\ar@{-->}[d] & {\bullet}\ar[d]^0\\
0\ar[r] & {\bullet}\ar@{=}[r] & {\bullet}\ar[r] & 0}
\eeq
\end{proof}

These weak (co)kernels are actually special cases of the more general notion of \emph{homotopy (co)limit}, which can not be discussed here in full generality; see~\cite{Dugger,Riehl} for recent expositions. We will come across another homotopy colimit (namely a homotopy pushout) in the proof of Proposition~\ref{trisubcatlocal}.
\end{rem}

We now start to study biproducts in triangulated categories, starting with a lemma.

\begin{lem}
\label{triangulatedbiproduct}
A not necessarily commutative diagram
\beq
\xymatrix@+5pt{X_1\ar@<1ex>[r]^{i_1} & Y\ar@<1ex>[l]^{p_1}\ar@<-1ex>[r]_{p_2} & X_2\ar@<-1ex>[l]_{i_2}}
\eeq
with a triangle
\beq
\xymatrix{X_1\ar[r]^{i_1} & Y\ar[r]^{p_2} & X_2\ar[r]^0 & {\Sigma X_1}}
\eeq
is a biproduct diagram if and only if the equations
\beq
p_1i_1=\id_{X_1},\qquad p_2i_2=\id_{X_2},\qquad p_1i_2=0,\qquad p_2i_1=0
\eeq
hold.
\end{lem}

\begin{proof}
The above equations follow easily from the standard biproduct equations
\beq
p_1i_1=\id_{X_1},\qquad p_2i_2=\id_{X_2},\qquad i_1p_1+i_2p_2=\id_Y,
\eeq
so that the ``only if'' is straightforward. Hence the remaining task is to show that the third biproduct equation
\beq
i_1p_1+i_2p_2=\id_Y
\eeq
follows from the equations above together with the specified triangle. A direct calculation using the above equations shows that $e:= i_1p_1+i_2p_2$ is an idempotent satisfying $i_1=ei_1$ and $p_2=p_2e$. Now apply the composition axiom to the valid equation $i_1=ei_1$:
\beq
\xymatrix@!{X_1\ar@(ru,lu)[rr]^{i_1}\ar[rd]^{i_1} & & Y\ar@(ru,lu)[rr]^{}\ar[rd]^{p_2} & & C_e\ar@(ru,lu)[rr]^{}\ar[rd]^{} & & {\Sigma X_2}\\
& Y\ar[ur]^e\ar[rd]^{p_2} & & X_2\ar[ru]^{}\ar[rd]^{} & & {\Sigma Y}\ar[ru]_{} &\\
& & X_2\ar@(rd,ld)[rr]^{}\ar[ru]^k & & {\Sigma X_1}\ar[ru]_{} & &}
\eeq
It follows that $kp_2=p_2e=p_2$. But now $p_2$ has a right inverse, so we can cancel it from the right and get $k=\id_{X_2}$, which shows that $C_e\cong 0$. By Corollary~\ref{iso}, $e$ is therefore an isomorphism. But now, the only idempotent isomorphism is the identity: we can cancel $e$ from $e^2=e$ and end up with $e=\id_Y$.
\end{proof}

The following result can be thought of as sharpening Corollary~\ref{iso}:

\begin{prop}
\label{split}
\begin{compactenum}
\item[(a)]
If $(f,f',f'')$ is a triangle and $f$ has a left inverse, then $f''=0$. If $f$ has a right inverse, then $f'=0$.
\item[(b)]
\label{splitbiproduct}
Any $f:X\lra Y$ with a left inverse is the injection $i_1$ of a biproduct $Y\cong X\oplus Z$, where actually $Z=C_f$.
\end{compactenum}
\end{prop}

\begin{proof}
\begin{compactenum}
\item[(a)]
We start with the first case. So assume that $f$ has a left inverse $g$, which can be completed to a triangle $(g,g',g'')$. Then an application of~\ref{T5} yields a diagram
\beq
\xymatrix@!{X\ar@(ru,lu)@{=}[rr]\ar[rd]^f & & X\ar@(ru,lu)[rr]^{g'}\ar[rd] & & C_g\ar@(ru,lu)[rr]^{k''}\ar[rd]^{g''} & & {\Sigma C_f}\\
& Y	\ar[ur]^g\ar[rd]^{f'} & & 0\ar[ru]\ar[rd] & & \Sigma Y\ar[ru]^{\Sigma f'} &\\
& & C_f\ar@(rd,ld)[rr]^{f''}\ar[ru] & &\Sigma X\ar[ru]^{\Sigma f} & &}
\eeq
where the object in the center can be taken to be $0$ by~\ref{T1}. Hence $f''=0$ by commutativity. Similarly, if we start with a $g$ having a right inverse, then we choose a right inverse $f$ and arrive at the same diagram, which also shows that $g'=0$. As an aside, note that $C_g\cong \Sigma C_f$ since $\Sigma^{-1} C_{k''}\cong 0$.
\item[(b)]
This is an elaboration on the previous proof. We claim that the conditions of Lemma~\ref{triangulatedbiproduct} are satisfied with
\[
i_1:=f,\qquad p_1:=g,\qquad i_2:=\Sigma^{-1}(g''k''^{-1}),\qquad p_2:=f'.
\]
Indeed, we have $gf=\id_X$ by assumption, $(\Sigma f')\circ g'' k''^{-1} = \id_{\Sigma C_f}$ by commutativity, as well as $f'f=0$ and $(\Sigma g)g''=0$ by Proposition~\ref{vanishcomp}.\qedhere
\end{compactenum}
\end{proof}

By Proposition~\ref{shift}, the first part of this is actually equivalent to the following more symmetrical statement: In a triangle $(f,f',f'')$, if $f'$ has a left inverse, then $f=0$. Dually, if $f'$ has a right inverse, then $f''=0$. In view of Remark~\ref{weakkernel}, these observations do not come as a surprise.

The second part of this proposition is a strong constraint on the structure of an additive category if it is supposed to allow a triangulation. For example, the category of vector spaces over a field $K$ whose dimension is not equal to one is an additive category. On the other hand, it does have left invertible morphisms whose codomain does not decompose into such a biproduct, such as the inclusion of a two-dimensional vector space into a three-dimensional vector space. Consequently, this category does not admit a triangulation no matter how we try to define $\Sigma$.

\section{The $3\!\!\times\!\!3$-lemma, biproduct triangles and triples of composable morphisms}

The $3\!\!\times\!\!3$-lemma that we prove now is a generalization of both the composition axiom and filling morphisms. It will allow for a general existence proof of biproducts, and is also useful in other contexts.

\begin{lem}[The $3\!\!\times\!\!3$-lemma]
\label{3x3lemma}
Given a diagram
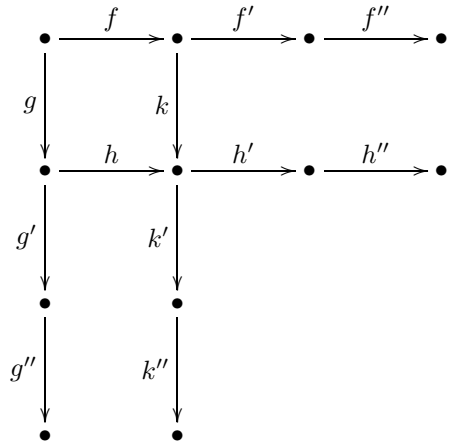
\beq
\xymatrix@!@+5pt{{\bullet}\ar[r]^f\ar[d]_g & {\bullet}\ar[r]^{f'}\ar[d]_k & {\bullet}\ar[r]^{f''} & {\bullet}\\
{\bullet}\ar[r]^h\ar[d]_{g'} & {\bullet}\ar[r]^{h'}\ar[d]_{k'} & {\bullet}\ar[r]^{h''} & {\bullet}\\
{\bullet}\ar[d]_{g''} & {\bullet}\ar[d]_{k''} & &\\
{\bullet} & {\bullet} & &}
\eeq
where the rows and columns are triangles, and the square is supposed to commute, it can be completed to a diagram
\beq
\xymatrix@!@+5pt{{\bullet}\ar[r]^(.47)f\ar[d]_(.47)g & {\bullet}\ar[r]^(.47){f'}\ar[d]_(.47)k & {\bullet}\ar[r]^(.47){f''}\ar@{-->}[d]_(.47)m & {\bullet}\ar[d]_(.47){\Sigma g}\\
{\bullet}\ar[r]^(.47)h\ar[d]_(.47){g'} & {\bullet}\ar[r]^(.47){h'}\ar[d]_(.47){k'} & {\bullet}\ar[r]^(.47){h''}\ar@{-->}[d]_(.47){m'} & {\bullet}\ar[d]_(.47){\Sigma g'}\\
{\bullet}\ar@{-->}[r]^(.47)j\ar[d]_(.47){g''} & {\bullet}\ar@{-->}[r]^(.47){j'}\ar[d]_(.47){k''} & {\bullet}\ar@{-->}[r]^(.47){j''}\ar@{-->}[d]_(.47){m''}\ar@{}[rd]|(.37)\circleddash & {\bullet}\ar[d]_(.47){\Sigma g''}\\
{\bullet}\ar[r]^(.47){\Sigma f} & {\bullet}\ar[r]^(.47){\Sigma f'} & {\bullet}\ar[r]^(.47){\Sigma f''} & {\bullet}}
\eeq
where the first three rows and columns are triangles, the last row and last column are negatives of triangles, and all squares commute, except for the bottom right square which anticommutes.
\end{lem}

\begin{proof}
As in the proof of Proposition~\ref{fillingmorphism}, complete the diagonal morphism $d:= kf=hg$ to a triangle $(d,d',d'')$:
\beq
\xymatrix@!@+5pt{{\bullet}\ar[r]^f\ar[d]_g\ar[rd]|d & {\bullet}\ar[r]^{f'}\ar[d]_(.47)k & {\bullet}\ar[r]^{f''} & {\bullet}\\
{\bullet}\ar[r]^(.47)h\ar[d]_{g'} & {\bullet}\ar[r]^{h'}\ar[d]_{k'}\ar[rd]|{d'} & {\bullet}\ar[r]^{h''} & {\bullet}\\
{\bullet}\ar[d]_{g''} & {\bullet}\ar[d]_{k''} & \circ\ar[rd]|{d''} & \\
{\bullet} & {\bullet} & & \circ\ar@(ld,rd)@{=}[lll]\ar@(ru,rd)@{=}[uuu]\\&}
\eeq
The two new objects are drawn as empty circles because they are not those objects which will appear in the final diagram at this place. Two applications of the composition axiom with respect to $d=hg$ and $d=kf$ yield triangles $(p,p',p'')$ and $(q,q',q'')$ drawn with dashed arrows as
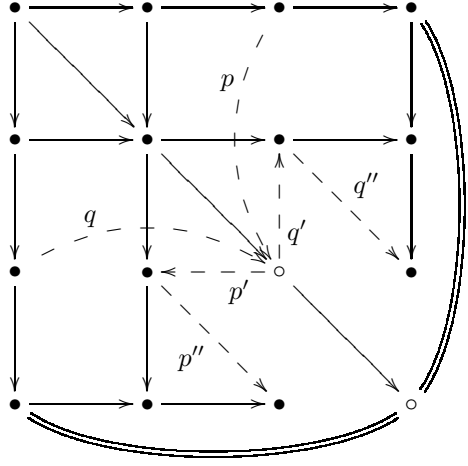
\beq
\xymatrix@!@+5pt{{\bullet}\ar[r]\ar[d]\ar[rd] & {\bullet}\ar[r]\ar[d] & {\bullet}\ar[r]\ar@/_1.4pc/@{-->}[dd]_(.3)p & {\bullet}\ar[d]\\
{\bullet}\ar[r]\ar[d] & {\bullet}\ar[r]\ar[d]\ar[rd] & {\bullet}\ar[r]\ar@{-->}[rd]^{q''} & {\bullet}\ar[d]\\
{\bullet}\ar[d]\ar@/^1.4pc/@{-->}[rr]^(.3)q & {\bullet}\ar[d]\ar@{-->}[rd]_{p''} & \circ\ar[rd]\ar@{-->}[u]_(.3){q'}\ar@{-->}[l]^(.3){p'} & {\bullet}\\
{\bullet}\ar[r] & {\bullet}\ar[r] & {\bullet} & \circ\ar@(ld,rd)@{=}[lll]\ar@(ru,rd)@{=}[uuu]\\&}
\eeq
where the whole diagram commutes and the unlabeled morphisms are those from the previous diagram. Now define $m:= q'p$, complete it to a triangle $(m,m',m'')$ using~\ref{T2}, and finally apply the composition axiom to the triangles $(p,p',p'')$, $(q',q'',-\Sigma q)$ and $(m,m',m'')$. This~\ref{T5} application 
\beq
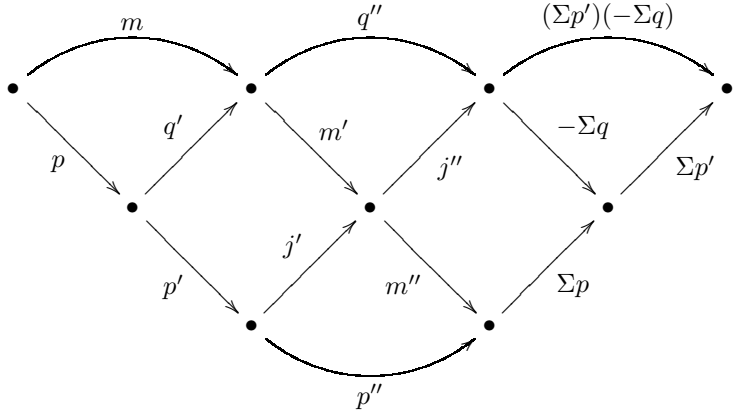
\xymatrix@!{{\bullet}\ar@(ru,lu)[rr]^m\ar[rd]_p && {\bullet}\ar@(ru,lu)[rr]^{q''}\ar[rd]^{m'} && {\bullet}\ar@(ru,lu)[rr]^{(\Sigma p')(-\Sigma q)}\ar[rd]^{-\Sigma q} && {\bullet}\\
& {\bullet}\ar[ru]^{q'}\ar[rd]_{p'} && {\bullet}\ar[ru]_{j''}\ar[rd]_{m''} && {\bullet}\ar[ru]_{\Sigma p'}\\
&& {\bullet}\ar[ru]^{j'}\ar@(rd,ld)[rr]_{p''} && {\bullet}\ar[ru]_{\Sigma p}}
\eeq
yields the two new morphisms $j'$ and $j''$ together with the commutativity relations
\beq
m'q'=j'p',\quad q''=j''m',\quad m''j'=p'',\quad -(\Sigma q)j''=(\Sigma p)m''
\eeq
and the assertion that $(j',j'',(\Sigma p')(-\Sigma q))$ forms a triangle. Combining these equations with the second to last diagram and the definition $j:= p'q$ gives a diagram
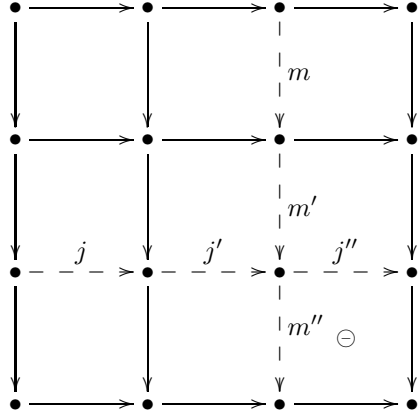
\beq
\xymatrix@!@+5pt{{\bullet}\ar[r]\ar[d] & {\bullet}\ar[r]\ar[d] & {\bullet}\ar[r]\ar@{-->}[d]^m & {\bullet}\ar[d]\\
{\bullet}\ar[r]\ar[d] & {\bullet}\ar[r]\ar[d] & {\bullet}\ar[r]\ar@{-->}[d]^{m'} & {\bullet}\ar[d]\\
{\bullet}\ar@{-->}[r]^j\ar[d] & {\bullet}\ar@{-->}[r]^{j'}\ar[d] & {\bullet}\ar@{-->}[r]^{j''}\ar@{-->}[d]^(.4){m''}\ar@{}[rd]|\circleddash & {\bullet}\ar[d]\\
{\bullet}\ar[r] & {\bullet}\ar[r] & {\bullet}\ar[r] & {\bullet}}
\eeq
for which (anti-)commutativity of the four lower right squares follows from the equations
\begin{align*}
j'k'=j'p'd'=m'q'd'=m'h', &\qquad j''m'=q''= (\Sigma g')h'', \\[10pt]
(\Sigma f')k''=p''=m''j', &\qquad (\Sigma f'')m'' = (\Sigma d'')(\Sigma p)m'' = -(\Sigma d'')(\Sigma q)j'' = - (\Sigma g'')m'',
\end{align*}
and finally $(j,j',j'')$ then is a triangle by functoriality of $\Sigma$ and Proposition~\ref{shift}.
\end{proof}

To understand the proof, it might also be helpful to consult~\cite[Lemma 2.6]{MayAdd} as an alternative explanation of the same steps. Also, note that the proof of Proposition~\ref{fillingmorphism} is in fact a special case of this, but considerably easier, so it makes sense to consider these two propositions separately.

\begin{rem}
Just like the Puppe sequence and the braid diagram, the $3\!\!\times\!\! 3$-lemma diagram can be extended indefinitely in both directions. It then forms an infinite lattice diagram
\beq
\xymatrix@!@+12pt{\ar@{}[dr]|{\ddots} & \ar@{}[d]|\vdots & {}\ar@{}[d]|\vdots & {}\ar@{}[d]|\vdots & {}\ar@{}[d]|\vdots & {}\ar@{}[d]|\vdots & {}\ar@{}[d]|\vdots & \ar@{}[dl]|\iddots\\{}
\ar@{}[r]|\ldots & {\bullet}\ar[r]^{\Sigma^{-2}j''}\ar[d]_{\Sigma^{-2}m''}\ar@{}[rd]|(.3)\circleddash & {\bullet}\ar[r]^{\Sigma^{-1}j}\ar[d]_{\Sigma^{-1}g''} & {\bullet}\ar[r]^{\Sigma^{-1}j'}\ar[d]_{\Sigma^{-1}k''} & {\bullet}\ar[r]^{\Sigma^{-1}j''}\ar[d]_{\Sigma^{-1}m''}\ar@{}[rd]|(.3)\circleddash & {\bullet}\ar[r]^j\ar[d]_{g''} & {\bullet}\ar[d]_{k''} & \ar@{}[l]|\ldots\\
\ar@{}[r]|\ldots&{\bullet}\ar[r]^{\Sigma^{-1}f''}\ar[d]_{\Sigma^{-1}m} & {\bullet}\ar[r]^f\ar[d]_g & {\bullet}\ar[r]^{f'}\ar[d]_k & {\bullet}\ar[r]^{f''}\ar[d]_m & {\bullet}\ar[r]^{\Sigma f}\ar[d]_{\Sigma g} & {\bullet}\ar[d]_{\Sigma k} & \ar@{}[l]|\ldots\\
\ar@{}[r]|\ldots&{\bullet}\ar[r]^{\Sigma^{-1}h''}\ar[d]_{\Sigma^{-1}m'} & {\bullet}\ar[r]^h\ar[d]_{g'} & {\bullet}\ar[r]^{h'}\ar[d]_{k'} & {\bullet}\ar[r]^{h''}\ar[d]_{m'} & {\bullet}\ar[r]^{\Sigma h}\ar[d]_{\Sigma g'} & {\bullet}\ar[d]_{\Sigma k'} &{}\ar@{}[l]|\ldots\\
\ar@{}[r]|\ldots&{\bullet}\ar[r]^{\Sigma^{-1}j''}\ar[d]_{\Sigma^{-1}m''}\ar@{}[rd]|(.3)\circleddash & {\bullet}\ar[r]^j\ar[d]_{g''} & {\bullet}\ar[r]^{j'}\ar[d]_{k''} & {\bullet}\ar[r]^{j''}\ar[d]_{m''}\ar@{}[rd]|(.3)\circleddash & {\bullet}\ar[r]^{\Sigma j}\ar[d]_{\Sigma g''} & {\bullet}\ar[d]_{\Sigma k''} &{}\ar@{}[l]|\ldots\\
\ar@{}[r]|\ldots&{\bullet}\ar[r]^{f''}\ar[d]_{m} & {\bullet}\ar[r]^{\Sigma f}\ar[d]_{\Sigma g} & {\bullet}\ar[r]^{\Sigma f'}\ar[d]_{\Sigma k} & {\bullet}\ar[r]^{\Sigma f''}\ar[d]_{\Sigma m} & {\bullet}\ar[r]^{\Sigma^2 f}\ar[d]_{\Sigma^2 g} & {\bullet}\ar[d]_{\Sigma^2 k} &{}\ar@{}[l]|\ldots\\
\ar@{}[r]|\ldots&{\bullet}\ar[r]^{h''} & {\bullet}\ar[r]^{\Sigma h} & {\bullet}\ar[r]^{\Sigma h'} & {\bullet}\ar[r]^{\Sigma h''} & {\bullet}\ar[r]^{\Sigma^2h} & {\bullet}&{}\ar@{}[l]|\ldots\\
\ar@{}[ur]|\iddots & {}\ar@{}[u]|\vdots & {}\ar@{}[u]|\vdots & {}\ar@{}[u]|\vdots & {}\ar@{}[u]|\vdots & {}\ar@{}[u]|\vdots & {}\ar@{}[u]|\vdots &{}\ar@{}[ul]|\ddots}
\eeq
where Puppe sequences go along the rows and columns, and all primitive squares commute, except for a sublattice of anticommuting primitive squares as indicated. The suspension functor acts in two ways, namely we can shift the diagram three squares to the right, or we can shift it three squares downwards. Thus the diagram repeats from the lower left to the upper right. Consequently, one can actually wrap the diagram around into a square tessellation of an infinite cylinder, where the cylinder axis extends from the upper left to the lower right. Then the suspension simply acts as a translation of the cylinder.
\end{rem}

We can now conclude the existence of biproducts, as announced earlier.

\begin{cor}
\label{biproductsexist}
Any pair of objects $X,Y\in\C$ has a biproduct $X\oplus Y$, and
\beq
\xymatrix{{}\Sigma^{-1}X\ar[r]^(.6)0 & Y\ar[r]^(.4){i_Y} & X\oplus Y\ar[r]^(.6){p_X} & X}
\eeq
is a triangle, where $i_Y$ and $p_X$ are the biproduct inclusion and projection, respectively.
\end{cor}

\begin{proof}
An application of the $3\!\!\times\!\!3$-lemma results in the diagram of triangles
\beq
\xymatrix{0\ar[r]\ar[d] & {}\Sigma^{-1}X\ar@{=}[r]\ar[d] & {}\Sigma^{-1}X\ar[r]\ar[d]^m & 0\ar[d]\\{}
\Sigma^{-1}Y\ar[r]\ar@{=}[d] & 0\ar[r]\ar[d] & Y\ar@{=}[r]\ar[d]^{i_Y} & Y\ar@{=}[d]\\{}
\Sigma^{-1}Y\ar[r]^j\ar[d] & X\ar[r]^{i_X}\ar@{=}[d] & Z\ar[r]_{p_Y}\ar[d]_{p_X} & Y\ar[d]\\
0\ar[r] & X\ar@{=}[r] & X\ar[r] & 0}
\eeq
where we have already named the morphisms in a suggestive manner. By commutativity, $p_Xi_X=\id_X$ and $p_Yi_Y=\id_Y$ are two of the biproduct relations, while $p_Yi_X=0=p_Xi_Y$ also hold by Proposition~\ref{vanishcomp}. The assertion then follows from Lemma~\ref{triangulatedbiproduct} and Proposition~\ref{shift} since $m=0$ by commutativity. The asserted triangle is identical to the third column of the diagram.
\end{proof}

\begin{prop}
\label{sumtriangles}
For any two triangles
\beq
\xymatrix{X_1\ar[r]^{f_1} & Y_1\ar[r]^{g_1} & Z_1\ar[r]^{h_1} & {}\Sigma X_1\\
X_2\ar[r]^{f_2} & Y_2\ar[r]^{g_2} & Z_2\ar[r]^{h_2} & {}\Sigma X_2}
\eeq
their direct sum
\beq
\xymatrix{X_1\oplus X_2\ar[rr]^{f_1\oplus f_2} && Y_1\oplus Y_2\ar[rr]^{g_1\oplus g_2} && Z_1\oplus Z_2\ar[rr]^(.45){h_1\oplus h_2} && {}\Sigma(X_1\oplus X_2)}
\eeq
is a triangle as well.
\end{prop}

\noindent Note that explicit mention of the natural isomorphism $\Sigma X_1\oplus\Sigma X_2\cong \Sigma(X_1\oplus X_2)$ has been omitted; see Remark~\ref{defrems}\ref{suspadd}.

\begin{proof}
There might be a proof of this using the $3\!\!\times\!\!3$-lemma, but it is actually easier to apply homological functors, the five lemma, and the Yoneda lemma in the form of Corollary~\ref{yonedacor}. This proof is very similar in spirit to the proof of Proposition~\ref{mappingcone}. By~\ref{T2}, construct a triangle
\beq
\xymatrix{X_1\oplus X_2\ar[rr]^{f_1\oplus f_2} && Y_1\oplus Y_2\ar[rr] && W\ar[rr] && {}\Sigma(X_1\oplus X_2)}
\eeq
Then for each pair of biproduct projection maps $j=1,2$, there is a filling morphism between triangles
\beq
\xymatrix{X_1\oplus X_2\ar[rr]^{f_1\oplus f_2}\ar[d] && Y_1\oplus Y_2\ar[rr]\ar[d] && W\ar[rr]\ar[d] && {}\Sigma(X_1\oplus X_2)\ar[d]\\
X_j\ar[rr]^{f_j} && Y_j\ar[rr] && Z_j\ar[rr] && \Sigma X_j}
\eeq
By the universal property of the biproduct as a product, these lift to
\beq
\xymatrix{X_1\oplus X_2\ar[rr]^{f_1\oplus f_2}\ar@{=}[d] && Y_1\oplus Y_2\ar[rr]\ar@{=}[d] && W\ar[rr]\ar[d] && {}\Sigma(X_1\oplus X_2)\ar@{=}[d]\\
X_1\oplus X_2\ar[rr]^{f_1\oplus f_2} && Y_1\oplus Y_2\ar[rr] && Z_1\oplus Z_2\ar[rr] && \Sigma (X_1\oplus X_2)}
\eeq
Under any homological functor, the first row has a long exact sequence since it is a triangle. The second row has a long exact sequence since it is a biproduct of triangles, so a homological functor maps it to a direct sum of long exact sequences. Hence the third vertical morphism also is an isomorphism by the five lemma and Corollary~\ref{yonedacor}.
\end{proof}

A special case of this is a shift of the triangle asserted Corollary~\ref{biproductsexist}: The biproduct of the pair of triangles
\beq
\xymatrix@R-15pt{X\ar[r] & 0 \ar[r] & \Sigma X\ar@{=}[r] & \Sigma X\\
0\ar[r] & Y\ar@{=}[r] & Y\ar[r] & 0}
\eeq
is given by the biproduct triangle
\beq
\xymatrix@+10pt{X\ar[r]^0 & Y\:\ar@{^{(}->}[r]^{i_Y} & \Sigma X\oplus Y\ar@{->>}[r]^(.57){p_{\Sigma X}} & \Sigma X}
\eeq

\begin{rem}
\label{nonuniquefm}
Now it is also possible to give an example (of a very general kind) of a non-unique filling morphism:
\beq
\xymatrix@+10pt{{}\Sigma^{-1}X\ar[rr]^0\ar@{=}[d] && Y\:\ar@{^{(}->}[rr]\ar@{=}[d] && X\oplus Y\ar@{->>}[rr]\ar[d]_{\left(\begin{array}{cc}\mathrm{id_X}&0\\f&\mathrm{id_Y}\end{array}\right)} && X\ar@{=}[d]\\{}
\Sigma^{-1}X\ar[rr]^0 && Y\:\ar@{^{(}->}[rr] && X\oplus Y\ar@{->>}[rr] && X}
\eeq
In this diagram, $f$ can be any element in $\C(X,Y)$ whatsoever, and the vertical morphism is a filling morphism between the two given triangles. Note that the matrix is always invertible, in accordance with Proposition~\ref{mappingcone}.
\end{rem}

We end our investigation of homological algebra in triangulated categories with a result on compositions of three morphisms. This is part of the axiom \textsc{Tr 5} proposed in~\cite{Schm}. We do not know whether Schmidt's \textsc{Tr 5} can be derived in full generality from the standard axioms~\ref{T1}--\ref{T5}.

\begin{prop}
\label{triplecomp}
Let
\[
\xymatrix{ \bullet \ar[r]^f & \bullet \ar[r]^g & \bullet \ar[r]^h & \bullet }
\]
be a triple of composable morphisms in $\C$ (not necessarily a candidate triangle). Then there is a diagram
\[
\xymatrix{ & C_{gf} \ar[dr] \\
C_f \ar[ur] \ar[rr] & & C_{hgf} \ar[r] & C_{hg} \ar[r] & \Sigma C_f }
\]
in which the horizontal morphisms form a triangle.
\end{prop}

\begin{proof}
We start by choosing mapping cone triangles for all the relevant morphisms $f$, $g$, $h$, $gf$, $hg$ and $hgf$ involved containing all the mapping cone objects appearing in the target diagram. Then we consider three applications of~\ref{T5}; in each case, we only name those morphisms that are of relevance to the argument, and write $i_\ast$ and $p_\ast$ for a mapping cone inclusion and projection, respectively, where $\ast$ is the morphism that we take the mapping cone of. First, for the composite $gf$, the~\ref{T5} application gives
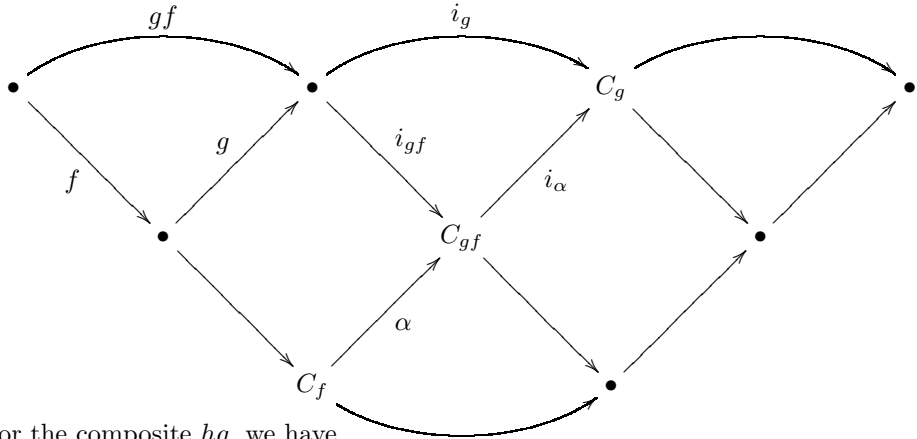
\[
\xymatrix@!{ {\bullet}\ar@(ru,lu)[rr]^{gf}\ar[rd]_f && {\bullet}\ar@(ru,lu)[rr]^{i_g} \ar[rd]^{i_{gf}} && C_g \ar@(ru,lu)[rr]\ar[rd] && {\bullet}\\
& {\bullet}\ar[ru]^g \ar[rd] && C_{gf} \ar[ru]_{i_\alpha} \ar[rd] && {\bullet}\ar[ru] \\
&& C_f \ar[ru]_\alpha \ar@(rd,ld)[rr] && \bullet \ar[ru] \\ }
\]
Second, for the composite $hg$, we have
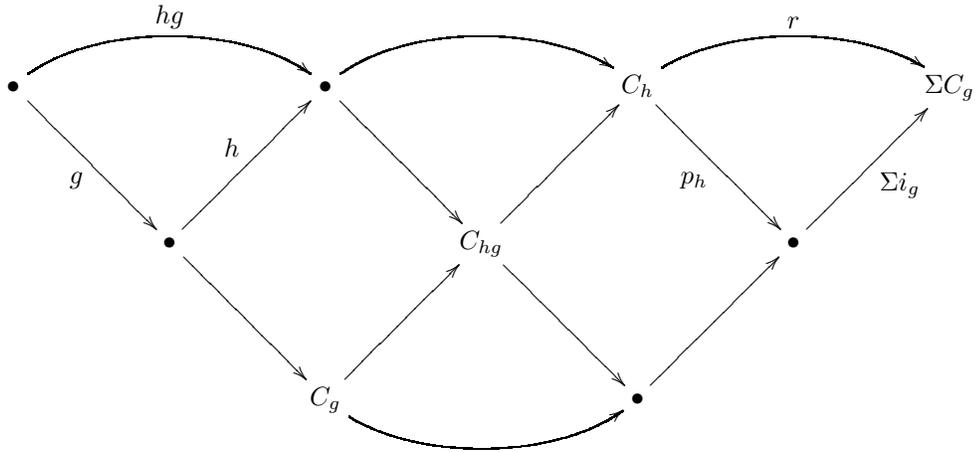
\[
\xymatrix@!{ {\bullet}\ar@(ru,lu)[rr]^{hg}\ar[rd]_g && {\bullet}\ar@(ru,lu)[rr] \ar[rd] && C_h \ar@(ru,lu)[rr]^r \ar[rd]_{p_h} && \Sigma C_g \\
& {\bullet}\ar[ru]^h \ar[rd] && C_{hg} \ar[ru] \ar[rd] && {\bullet}\ar[ru]_{\Sigma i_g} \\
&& C_g \ar[ru] \ar@(rd,ld)[rr] && \bullet \ar[ru] \\}
\]
And third, for the composite $h(gf)$, we can assume
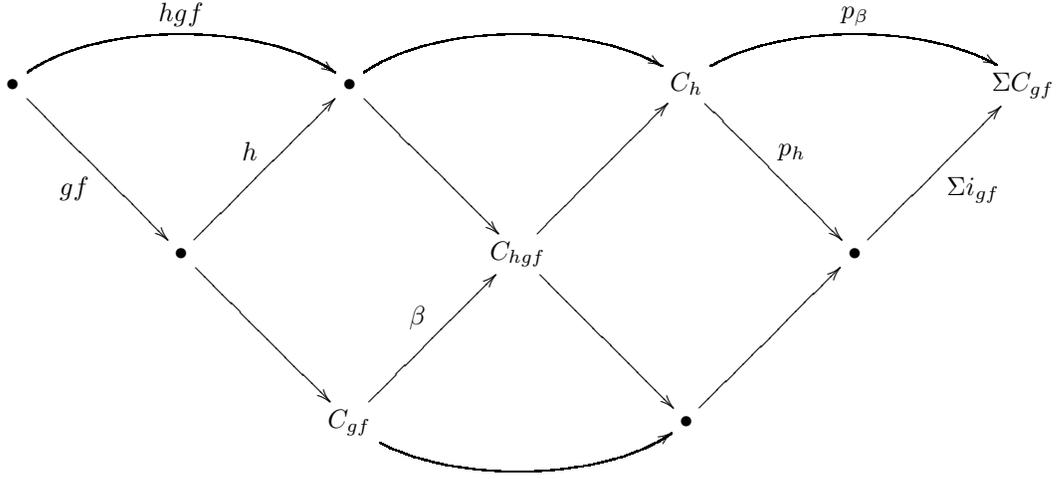
\[
\xymatrix@!{ {\bullet}\ar@(ru,lu)[rr]^{hgf}\ar[rd]_{gf} && {\bullet}\ar@(ru,lu)[rr] \ar[rd] && C_h \ar@(ru,lu)[rr]^{p_\beta} \ar[rd]^{p_h} && \Sigma C_{gf} \\
& {\bullet}\ar[ru]^h \ar[rd] && C_{hgf} \ar[ru] \ar[rd] && {\bullet}\ar[ru]_{\Sigma i_{gf}} \\
&& C_{gf} \ar[ru]^\beta \ar@(rd,ld)[rr] && \bullet \ar[ru] }
\]\\

\noindent Note that some of them appear in two of these diagrams. We can now consider the composite $\beta\alpha\: :\: C_f\to C_{hgf}$ and use another instance of~\ref{T5}, this time with respect to the composition $\beta\alpha$ together with the identifications $C_\alpha \cong C_g$ and $C_\beta \cong C_h$ which are part of the previous diagrams.
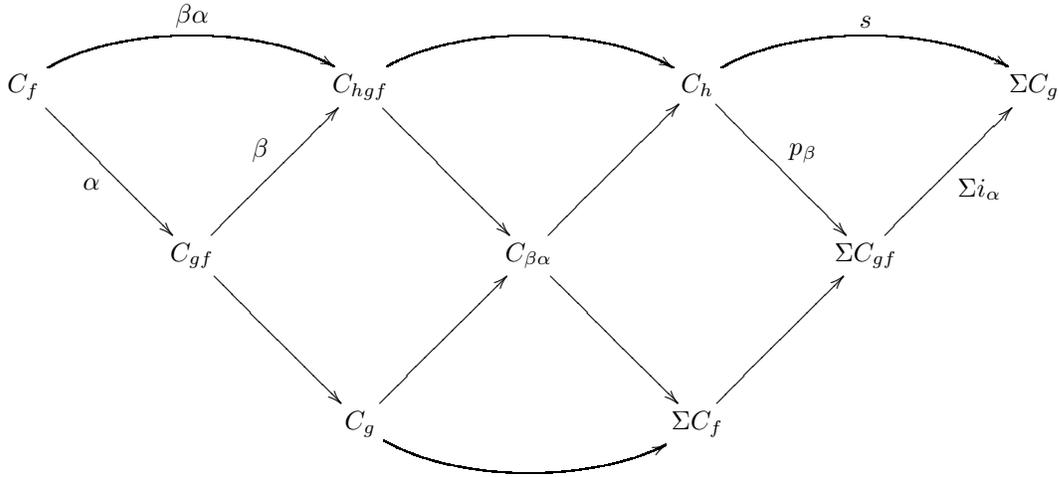
\[
\xymatrix@!{ C_f \ar@(ru,lu)[rr]^{\beta\alpha} \ar[rd]_{\alpha} && C_{hgf} \ar@(ru,lu)[rr] \ar[rd] && C_h \ar@(ru,lu)[rr]^s \ar[rd]^{p_\beta} && \Sigma C_g \\
& C_{gf} \ar[ru]^\beta \ar[rd] && C_{\beta\alpha} \ar[ru] \ar[rd] && \Sigma C_{gf} \ar[ru]_{\Sigma i_\alpha} \\
&& C_{g} \ar[ru] \ar@(rd,ld)[rr] && \Sigma C_f \ar[ru] }
\]
\\

\noindent Our goal is to identify $C_{\beta\alpha}\cong C_{hg}$. This is based on showing that $s=r$, which implies that $C_{\beta\alpha} \cong \Sigma^{-1} C_s \cong \Sigma^{-1} C_r \cong C_{hg}$. But this indeed holds true, since 
\[
s = (\Sigma i_\alpha) p_\beta = (\Sigma i_\alpha) (\Sigma i_{gf}) p_h = \Sigma(i_\alpha i_{gf}) p_h = (\Sigma i_g) p_h = r. \qedhere
\]
\end{proof}

\section{Triangulated functors and triangulated subcategories}
\label{verdier}

Given triangulated categories $\D$ and $\C$ with respective suspensions $\Sigma_D$ and $\Sigma_C$, when can we say that a functor $F:\D\lra\C$ preserves the structure, namely the triangulation? Certainly it has to preserve the suspension and map triangles to triangles, both in an appropriate sense. It would be too strong a condition to require the commutativity relation $F\Sigma_D=\Sigma_C F$ on the nose. Instead, it should be sufficient for the two compositions to be merely naturally isomorphic via a natural transformation $\eta:F\Sigma_D\stackrel{\sim}{\lra} F\Sigma_C$. But then, we cannot expect $F$ to preserve triangles on the nose, since $F$ does not even map candidate triangles to candidate triangles in the strict sense. So, also the triangles need to be suitably dealt with using $\eta$:

\begin{defn}
A functor $F:\D\lra\C$ between triangulated categories together with a natural isomorphism $\eta:F\Sigma_D\stackrel{\sim}{\lra} F\Sigma_C$ is called a \emph{triangulated functor} if for any triangle
\beq
\xymatrix{X\ar[r]^f & Y\ar[r]^g & Z\ar[r]^(.4)h & {}\Sigma_DX}
\eeq
the horizontal part of
\beq
\xymatrix{F(X)\ar[r]^{F(f)} & F(Y)\ar[r]^{F(g)} & F(Z)\ar[rr]\ar[rd]_{F(h)} && {}\Sigma_CF(X)\\
&&&F(\Sigma_DX)\ar[ru]_{\eta_X}^\sim &}
\eeq
is a triangle in $\C$.
\end{defn}

Formally, $\eta$ is part of the data of a triangulated functor. On the other hand, there is usually an obvious choice for what $\eta$ is, and one typically leaves it unmentioned in concrete cases. In the case of the canonical inclusion functor for a triangulated subcategory (see the next definition), it actually is the identity transformation anyway, since the inclusion functor \emph{does} preserve $\Sigma$ on the nose.

\begin{defn}
\label{trisubcat}
A full subcategory $\D\subseteq\C$ is called a \emph{triangulated subcategory} if it is closed under (de)suspension and itself satisfies the axioms~\ref{T1} to~\ref{T5}, where the triangles are exactly those candidate triangles which are also triangles in $\C$.
\end{defn}

By the fullness premise, a triangulated subcategory is determined by the class of its objects. Thus, we can---and will---freely confuse a triangulated subcategory with the class of its objects. In this formulation, a triangulated subcategory is a class of objects closed under suspension and desuspension, and containing some mapping cone for every morphism between any two of its objects:

\begin{lem}
\label{trisubcatcrit}
A class of objects in $\C$ is a triangulated subcategory if and only if it is closed under (de)suspension and contains at least one mapping cone for any morphism between any two of its objects.
\end{lem}

\begin{proof}
The ``only if'' part is clear. In the ``if'' direction, we need to show that~\ref{T1}--\ref{T5} hold. \ref{T3} and~\ref{T4} hold automatically since the triangulation is the one inherited from $\C$, while~\ref{T1},~\ref{T2} and~\ref{T5} follow from the fullness assumption together with existence of mapping cones; concerning~\ref{T1}, the zero object in the subcategory is guaranteed to existence as the mapping cone $\id_X$, although it might not coincide ``on the nose'' with the chosen zero object of $\C$.
\end{proof}

\begin{defn}
A triangulated subcategory $\D\subseteq\C$ is called a \emph{thick triangulated subcategory}, if the following condition holds: whenever $Y\in\D$ and there is a left invertible $f:X\lra Y$ in $\C$, then $X\in\D$.
\end{defn}

In  particular, since every isomorphism has a left inverse, a thick triangulated subcategory is closed under isomorphism of objects. Moreover, Proposition~\ref{splitbiproduct} implies that $\D$ is thick if and only if the following condition holds: whenever $\D$ contains an object isomorphic to $X\oplus Y$, then it also contains $X$ and $Y$. This is the condition that one most frequently uses in practice to verify that a triangulated subcategory is thick.

\begin{ex}
\label{acyclic}
Let $H:\C\lra\Ab$ be a homological functor. Then an object $X\in\C$ is called $H$-\emph{acyclic} if $H(\Sigma^nX)=0$ for all $n\in\Z$. The long exact sequences for $H$ show that the class of $H$-acyclic objects is a thick triangulated subcategory of $\C$.
\end{ex}

For any triangulated subcategory $\D\subseteq\C$, the smallest thick triangulated subcategory which contains $\D$ is called the \emph{thick closure} of $\D$ and denoted by $\overline{\D}$; since the intersection of a collection of thick triangulated subcategories is again a thick triangulated subcategory, it is clear that such a smallest $\overline{\D}$ indeed exists. It is also clear that the assignment $\D\mapsto\overline{\D}$ is a closure operation.

\begin{prop}
\label{thickclosure}
$\overline{\D}$ has as objects exactly those $X$ such that $X\oplus\Sigma X$ is isomorphic to an object in $\D$.
\end{prop}

\begin{proof}
Clearly every such object has to lie in $\overline{\D}$. Hence it remains to show that the class of these objects forms a thick triangulated subcategory. This will be done in two steps:
\begin{compactenum}
\item Show that the class of these objects is a triangulated subcategory. It closed under suspension and desuspension since $\D$ is. Hence by Lemma~\ref{trisubcatcrit}, it remains to check that it is closed under mapping cones. So suppose $X\oplus\Sigma X$ and $Y\oplus\Sigma Y$ are both isomorphic to objects in $\D$, and let $f\in\C(X,Y)$ be a morphism with mapping cone $C_f$. Then by Propositions~\ref{sumtriangles}, there is a triangle
\beq
\xymatrix@+3pc{X\oplus\Sigma X\ar[r]^{\left(\begin{array}{cc}f&0\\0&-\Sigma f\end{array}\right)} & Y\oplus\Sigma Y\ar[r] & C_f\oplus\Sigma C_f\ar[r] & \Sigma(X\oplus\Sigma X)}
\eeq
which shows that $C_f\oplus\Sigma C_f$ also is isomorphic to an object in $\D$ since $\D$ is closed under mapping cones.
\item Checking thickness is just as easy. Suppose that the alleged class contains the object $X\oplus Y$, which means that $X\oplus Y\oplus\Sigma X\oplus\Sigma Y$ is isomorphic to an object in $\D$. Now since $\D$ is triangulated, it also contains a mapping cone of the morphism
\beq
\xymatrix@+9pc{X\oplus Y\oplus\Sigma X\oplus\Sigma Y\ar[r]^{\left(\begin{array}{cccc}0&0&0&0\\0&\id_Y&0&0\\0&0&\id_{\Sigma X}&0\\0&0&0&\id_{\Sigma Y}\end{array}\right)} & X\oplus Y\oplus \Sigma X\oplus\Sigma Y}
\eeq
On the other hand, such a mapping cone is isomorphic to the mapping cone of $X\stackrel{0}{\lra} X$, which is $X\oplus\Sigma X$. Therefore $X$ also lies in the purported class of objects.\qedhere
\end{compactenum}
\end{proof}

\begin{cor}
\label{thickness}
A triangulated subcategory $\D$ is thick if and only if the following condition holds: whenever $\D$ contains an object isomorphic to some $X\oplus \Sigma X$, then $X\in\D$.
\end{cor}

\begin{proof}
This follows immediately from the previous proposition.
\end{proof}

Similarly to forming the thick closure of a triangulated subcategory, we may want to consider the smallest (thick) triangulated subcategory containing a given class of objects in $\C$. Again, it is clear that such a smallest class exists: it is simply the intersection of all such (thick) triangulated subcategories containing the given class of objects. This will be called the (thick) triangulated subcategory \emph{generated} by the given class of objects.

The following result is sometimes useful for showing that an object belongs to a thick triangulated subcategory:

\begin{cor}
\label{thickcrit}
Let $\D\subseteq\C$ be a thick triangulated subcategory and $f\: :\: X\lra Y$ and $g\: :\: Y\lra Z$ a composable pair of morphisms with $Y\in\D$ and $C_{gf}\in\D$. Then also $X,Z\in\D$.
\end{cor}

\begin{proof}
Completing $f$, $g$ and $h:=gf$ to triangles results in the composition axiom diagram
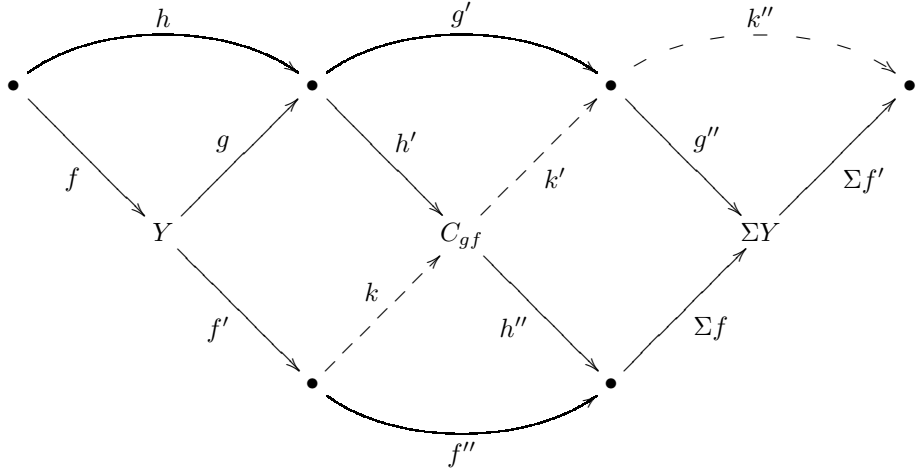
\[
\xymatrix@!{{\bullet}\ar@(ru,lu)[rr]^h\ar[rd]_f & & {\bullet}\ar@(ru,lu)[rr]^{g'}\ar[rd]^{h'} & & {\bullet}\ar@(ru,lu)@{-->}[rr]^{k''}\ar[rd]^{g''} & & {\bullet}\\
& Y \ar[ur]^g\ar[rd]_{f'} & & C_{gf} \ar@{-->}[ru]_{k'}\ar[rd]_{h''} & & \Sigma Y \ar[ru]_{\Sigma f'} &\\
& & {\bullet}\ar@(rd,ld)[rr]_{f''}\ar@{-->}[ru]^{k} & & {\bullet}\ar[ru]_{\Sigma f} & &}
\]
where we have only explicitly labelled those objects that we already know to be in $\D$. Since $C_h\in\D$ by assumption and $C_{g'}\cong\Sigma Y\in\D$ as well, we have $C_{g'h}\in\D$ by another application of~\ref{T5} to the composite $g'h$. On the other hand, we also have $g'h=g'gf=0$, and hence contains the mapping cone of a zero morphism with domain $X$. Since such a mapping cone necessarily contains $X$ as a direct summand, we have $X\in\D$ since $\D$ is thick. Finally, the triangle $(h,h',h'')$ proves that $Z\in\D$ as well.
\end{proof}

\section{Verdier localization}

In some situations, one encounters a triangulated category $\C$ and a class $\W$ of morphisms in $\C$ which one would like to be isomorphisms, although not all of them are. This leads to the following localization problem:

Find a triangulated category $\C[\W^{-1}]^\triangle$ and a triangulated functor $\Loc^\triangle:\C\lra\C[\W^{-1}]^\triangle$ having the following universal property:
\begin{compactenum}
\item $\Loc^\triangle(w)$ is an isomorphism for all $w\in\W$,
\item\label{localizeb} If $F:\C\lra\D$ is any triangulated functor which also maps $\W$ to isomorphisms, then $F$ factors uniquely over $\Loc^\triangle$ as in the diagram
\beq
\xymatrix{{}\C\ar[rr]^{\Loc^\triangle}\ar[rd]_F & & {}\C[\W^{-1}]^\triangle\ar@{-->}[ld]^{!} \\
& {}\D & }
\eeq
\end{compactenum}
As one should expect for a universal property of categories, the diagram in~\ref{localizeb} can only be required to commute up to natural isomorphism, and also the uniqueness of the induced functor can only postulated up to natural isomorphism. This implies that the pair consisting of the triangulated category $\C[\W^{-1}]^\triangle$ and the localization functor $\Loc^\triangle$ is unique up to equivalence of triangulated categories (if it exists).

In such a situation, $\C[\W^{-1}]^\triangle$ is called the \emph{Verdier localization}~\cite{Ver} of $\C$ with respect to $\W$. In general, the Verdier localization $\C[\W^{-1}]^\triangle$ is different from the ordinary localization $\C[\W^{-1}]$ having the same universal property, but with respect to plain categories instead of triangulated categories. For the ordinary localization $\C[\W^{-1}]$, and also for the terminology used in conjunction with localization and categories of fractions such as the axioms (L0), (L1), and so on, we refer to~\cite{Fri1}.

We will now see how the problem of existence of a Verdier localization relates to triangulated subcategories via the correspondence of Corollary~\ref{iso} between isomorphisms and zero objects. So consider some triangulated functor $F:\D\lra\C$ such that $F(w)$ is invertible for some $w\in\W$. By the triangle
\beq
\xymatrix{{\bullet}\ar[r]^{F(w)} & {\bullet}\ar[r] & F(C_w)\ar[r] & {\bullet}}
\eeq
we have that $F(C_w)\cong 0$ for any mapping cone $C_w$ of $w$. Conversely, if $F(C_w)\cong 0$, then $F(w)$ must be an isomorphism. Hence, instead of considering those morphisms which $F$ maps to isomorphisms, we may equivalently consider the class of those objects which $F$ maps to zero objects, which is the kernel of $F$:

\begin{defn}
The \emph{kernel of a triangulated functor} $F:\D\lra\C$ is the thick triangulated subcategory $\mathrm{ker}(F)\subseteq\C$ containing all those objects $X$ for which $F(X)\cong 0$ in $\C$.
\end{defn}

A priori, $\mathrm{ker}(F)$ is only a class of objects in $\C$, but one can easily check that $\mathrm{ker}(F)$ actually is a thick triangulated subcategory. So given $\W$ and assuming that the Verdier localization exists, we know that the kernel of the localization functor $\Loc^\triangle$ contains the thick triangulated subcategory generated by the mapping cones $\{C_w,\,w\in\W\}$. By constructing Verdier localizations as categories of fractions, we will see that any thick triangulated subcategory does indeed occur in this fashion as the kernel of some triangulated localization functor $\Loc^\triangle$ for an appropriate $\W$.

This leads to the observation that there are two ways to specify a Verdier localization: either, specify the class of morphisms $\W$ which shall become isomorphisms; or, specify a triangulated subcategory $\D\subseteq\C$ of objects which shall become zero, and then all the morphisms in
\beq
\Iso(\D):=\{\, f\in\Mor(\C)\:|\:\textrm{There is a triangle}\xymatrix{{\bullet}\ar[r]^f & {\bullet}\ar[r] & C_f\ar[r] & {\bullet}}\textrm{ with } C_f\in\D\,\},
\eeq
become invertible. Since we have already acquired some understanding of triangulated subcategories, we start discussing Verdier localizations with respect to these, and then get back to describing $\C[\W^{-1}]^\triangle$ for arbitary $\W$ afterwards. So from now on, we fix a triangulated subcategory $\D\subseteq\C$ and consider Verdier localization with respect to $\W:=\Iso(\D)$. In the following, morphisms in this localizing class are drawn in diagrams as wiggly arrows. 

\begin{lem}
\label{2of3}
$\Iso(\D)$ has the \emph{2-out-of-3 property}: if $v,w\in\Mor(\C)$ is a composable pair of morphisms, and if two of $v$, $w$, and $wv$ are in $\Iso(\D)$, then so is the third.
\end{lem}

\begin{proof}
All three parts of this statement follow from~\ref{T5} and the assumption that $\mathcal{D}$ is a triangulated subcategory: the triangle that one obtains from applying~\ref{T5} to the composition $wv$ yields a triangle on the objects $C_v$, $C_{wv}$ and $C_w$, which shows that if $\D$ contains objects isomorphic to two of these, then it also contains an object isomorphic to the third.
\end{proof}

\begin{prop}
\label{trisubcatlocal}
If $\D$ is a triangulated subcategory of $\C$, then $\C$ allows a calculus of left and right fractions with respect to $\Iso(\D)$.
\end{prop}

See~\cite{Fri1} for detailed explanations of the categories of fractions terminology and axioms used in the proof.

\begin{proof} It is sufficient to show this for the calculus of left fractions; the other part then follows by duality. 

Axiom (L0) states that $\Iso(D)$ is closed under composition and contains all identities. While the first is part of Lemma~\ref{2of3}, the second is trivial by~\ref{T1}. Concerning (L1), we have to start with a diagram
\beq
\xymatrix{X\ar[r]^f\ar@{~>}[d]_w & Y\\
X' &}
\eeq
with $w\in\Iso(D)$, and show that it can completed to a commutative square with the right vertical morphism also in $\Iso(D)$. By means of~\ref{T2}, choose any triangle going over $X'\oplus Y$ via $w$ and $f$ as in
\beq
\xymatrix@+10pt{X\ar[rr]^{\left(\begin{array}{c}w\\f\end{array}\right)} && X'\oplus Y\ar[rr]^{\left(\begin{array}{cc}-f'&w'\end{array}\right)} && Y'\ar[rr] && {}\Sigma X}
\eeq
and define $Y'$, $f'$ and $w'$ as indicated by the labels. Then $-f'w+w'f=0$ by Proposition~\ref{vanishcomp}, so that
\beq
\xymatrix{X\ar[r]^f\ar@{~>}[d]_w & Y\ar@{~>}[d]_{w'}\\
X'\ar[r]^{f'} & Y'}
\eeq
commutes. In the literature, this construction is known as ``taking the homotopy pushout'' and is a special case of a general notion of ``homotopy colimit'' smilar to Remark~\ref{weakkernel}. 
It still needs to be shown that $w'\in\Iso(\D)$. For this, consider the composition axiom diagram
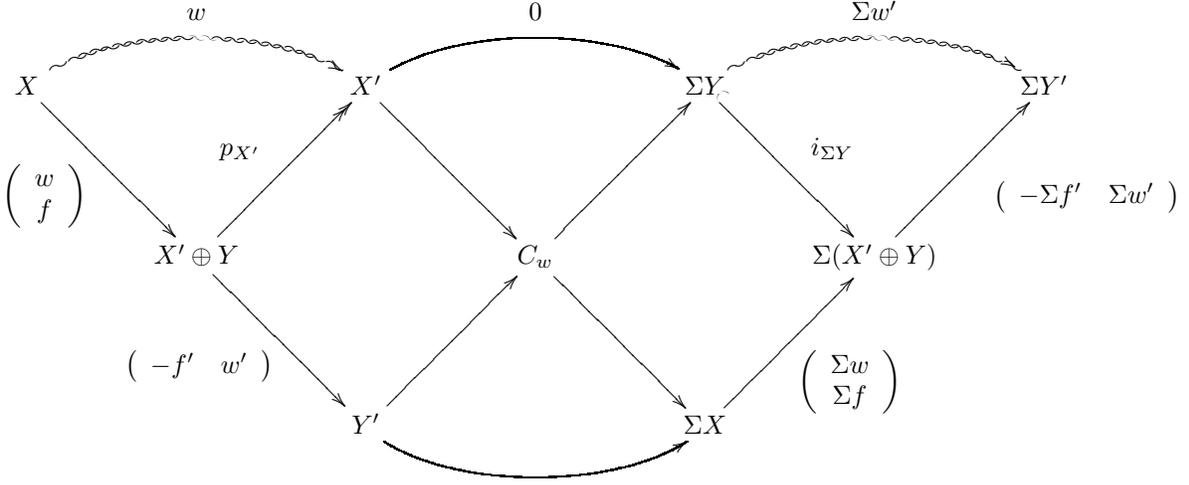
\beq
\xymatrix@!@-23pt{X\ar@(ru,lu)@{~>}[rr]^{\begin{array}{c}w\end{array}}\ar[rd]_(.4){\left(\begin{array}{c}w\\f\end{array}\right)} && X'\ar@(ru,lu)[rr]^{\begin{array}{c}0\end{array}}\ar[rd] && {}\Sigma Y\ar@(ru,lu)@{~>}[rr]^{\begin{array}{c}\Sigma w'\end{array}}\ar@{^{(}->}[rd]^{\begin{array}{c}i_{\Sigma Y}\end{array}} && {}\Sigma Y'\\
&X'\oplus Y\ar@{->>}[ru]^{\begin{array}{c}p_{X'}\end{array}}\ar[rd]_{\left(\begin{array}{cc}-f'&w'\end{array}\right)} && C_w\ar[rd]\ar[ru]&&{}\Sigma(X'\oplus Y)\ar[ru]_{\quad\left(\begin{array}{cc}-\Sigma f'&{}\Sigma w'\end{array}\right)} &\\
&& Y'\ar@(rd,ld)[rr]\ar[ru] && {}\Sigma X\ar[ru]_{\left(\begin{array}{c}\Sigma w\\{}\Sigma f\end{array}\right)}&&}
\eeq\\

\noindent which shows that $C_{\Sigma w'}\cong\Sigma C_w$. Hence we get $C_{w'}\cong C_w$ by Remark~\ref{shiftcone}, and consequently $w'\in\Iso(\D)$.

For Axiom (L2''), we need to assume that we are given a diagram
\beq
\xymatrix{X\ar@{~>}[r]^w & X'\ar[r]^f & Y}
\eeq
with $w\in\Iso(\D)$ such that $fw=0$ holds, and we need to find a morphism $u$ with $uf=0$ such that $u$ lies in the subcategory generated by $\Iso(\D)$ and all left invertible morphisms. To see this, apply the composition axiom to the equation $fw=0$ in order to obtain a morphism $v$ as in
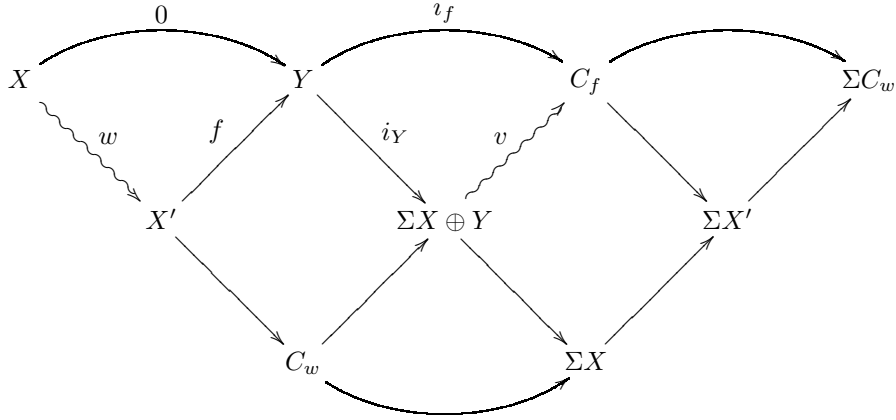
\beq
\xymatrix@!@-23pt{X\ar@(ru,lu)[rr]^0\ar@{~>}[rd]^w && Y\ar@(ru,lu)[rr]^{i_f}\ar[rd]^{i_Y} && C_f\ar@(ru,lu)[rr]\ar[rd] && {}\Sigma C_w\\
& X'\ar[rd]\ar[ru]^f && \Sigma X\oplus Y\ar[rd]\ar@{~>}[ru]^{v} && \Sigma X'\ar[ru] &\\
&& C_w\ar@(rd,ld)[rr]\ar[ru] && \Sigma X\ar[ru]&&}
\eeq
\\

Since $C_v\cong \Sigma C_w$, we have $v\in\Iso(\D)$. Moreover, we have $i_f f = 0$. In order for this to be the data required by Axiom (L2''), it remains to be shown that $i_f$ lies in the subcategory generated by $\Iso(\D)$ and all left invertible morphisms. But this is the case because $i_f=v i_Y$, and $v\in\Iso(\D)$, and the biproduct inclusion $i_Y$ has a left inverse given by the corresponding biproduct projection.
\end{proof}

In the resulting category of fractions, it may happen that the class of zero objects is strictly larger than the class of objects of $\D$. For example, $\C$ may contain objects isomorphic to objects in $\D$ which are not themselves contained in $\D$. As a less trivial example, take $\C$ to be the category of finite-dimensional vector spaces over a field $K$ with the triangulation as described in Example~\ref{semisimpleab}. Let $\D$ be the triangulated subcategory consisting of all even-dimensional vector spaces. For any linear map between even-dimensional vector spaces, the sum of the dimensions of the kernel and of the cokernel is also even, which implies that $\D$ is indeed closed under mapping cones. Now consider the triangle
\beq
\xymatrix{K\ar[rr]^{\begin{array}{c}0\end{array}} && K\ar[rr]^{\left(\begin{array}{c}1\\0\end{array}\right)} && K\oplus K\ar[rr]^(.55){\left(\begin{array}{cc}0&1\end{array}\right)} && K}
\eeq
This triangle expresses the fact that $0:K\lra K$ is in $\Iso(\D)$, so that $K$ becomes zero in the localization. In fact, this shows that $\C[\Iso(\D)^{-1}]$ is trivial in the sense that \emph{all} objects are zero, although $\D$ was a proper subcategory. As one might already be able to guess, the crux of the matter is that this $\D$ is not thick, since it contains $K\oplus K$, but not $K$.

In general, whenever $Y\in\D$ and there is a morphism $f:X\lra Y$ with left-inverse $g$, then $X$ will also be trivial in the localization, since there we have that $\mathrm{id_X}=gf=g\circ\id_Y\circ f=g\circ 0\circ f=0$. We will see in Proposition~\ref{verdierkernel} below that this is all that can happen.

We also write $\C/\D$ for the resulting category of fractions $\C[\Iso(\D)^{-1}]$, and $\Loc^\triangle \: : \: \C\lra \C/\D$ for the localization functor. Before showing that $\C/\D$ inherits a triangulation from $\C$, we need a bit more preparation, using Proposition~\ref{trisubcatlocal}.

\begin{lem}
\label{fraceq}
For $f,g\in\C(X,Y)$, we have $\Loc(f)=\Loc(g)$ in $\C/\D$ if and only if there exists $w\in\Iso(\D)$ such that $wf=wg$ in $\C$. Or, equivalently, if and only if there exists $v\in\Iso(\D)$ such that $fv=gv$.
\end{lem}

As the following proof shows, the first statement actually holds in any category of left fractions, while the second holds in any category of right fractions.

\begin{proof}
By duality, it is enough to prove the first statement.

By definition of the category of fractions, the roofs $(f,\id_Y)$ and $(g,\id_Y)$ (or formal fractions $\id_Y^{-1}\circ f$ and $\id_Y^{-1}\circ g$) represent the same morphism in $\C/\D$ if and only if there exist $h,k\in\Mor(\C)$ fitting into a diagram
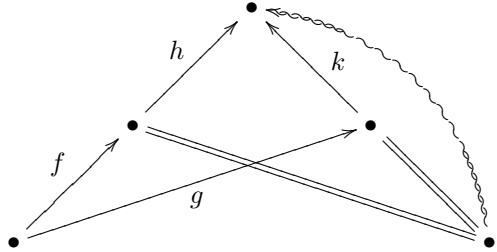
\[
	\xymatrix{ & & \bullet \\
		& \bullet \ar[ur]^h & & \bullet \ar[ul]_(.4)k \\
		\bullet \ar[ur]^f \ar[urrr]_g & & & & \bullet \ar@{=}[ul] \ar@{=}[ulll] \ar@{~>}@/_2pc/[uull] }
\]
where commutativity implies that $h=k\in\Iso(\D)$. This is exactly the condition claimed.
\end{proof}

The following criterion for becoming an isomorphism in $\C/\D$ is reminiscent of the \emph{2-out-of-6 property}~\cite{DHK}.

\begin{lem}
\label{isocrit1}
For $f\in\Mor(\C)$, $\Loc(f)$ is an isomorphism in $\C/\D$ if and only if there exist $g,h\in\Mor(\C)$ such that $gf\in\Iso(\D)$ and $fh\in\Iso(\D)$.
\end{lem}

As the following proof shows, this actually holds in any category of (left and right) fractions.

\begin{proof}
We start with the ``if'' direction. If $gf\in\Iso(\D)$ and $fh\in\Iso(\D)$, then this means that both $\Loc(g)\Loc(f)$ and $\Loc(f)\Loc(h)$ are isomorphisms. Since the class of isomorphisms has the 2-out-of-6 property~\cite{DHK}, it follows that $\Loc(f)$ is an isomorphism as well.

Conversely, suppose that $\Loc(f)$ is an isomorphism. Then $\Loc(f)$ has an inverse in $\C/\D$, which we can write either as a left fraction $\Loc(v)^{-1}\Loc(g)$ or as a right fraction $\Loc(h)\Loc(w)^{-1}$ with $v,w\in\Iso(\D)$. By virtue of being an inverse, we have that
\[
\Loc(v)^{-1}\Loc(g)\Loc(f) = \id,\qquad \Loc(f)\Loc(h)\Loc(w)^{-1} = \id,
\]
which we can rewrite as
\[
\Loc(gf) = \Loc(v),\qquad \Loc(fh) = \Loc(w).
\]
By Lemma~\ref{fraceq}, we can lift this to an equation in $\C$ by post- or precomposing with a suitable element of $\Iso(\D)$. By absorbing the new element of $\Iso(\D)$ into $g$ and $v$ or $h$ and $w$, respectively, we obtain equations
\[
gf = v,\qquad fh = w,
\]
as has been claimed.
\end{proof}

We can now prove the previously indicated relationship to thick triangulated subcategories:

\begin{prop}
\label{isocrit2}
For any $f\in\Mor(\C)$, $\Loc(f)$ is an isomorphism in $\C/\D$ if and only if $f\in\Iso(\overline{\D})$.
\end{prop}

\begin{proof}
We start with the ``if'' direction. The assumption $f\in\Iso(\overline{\D})$ for $f\: : X\lra Y$ means that $C_f$ is a direct summand of an object isomorphic to an object in $\D$. So let $Z\in\C$ be an object which makes $C_f\oplus Z$ isomorphic to an object in $\D$. Then consider the diagram of triangles
\[
\xymatrix@C+33pt@R+15pt{ X\oplus \Sigma^{-1}Z \ar[r]^{\left(\begin{array}{cc}f & 0\end{array}\right)} \ar[d]^{p_X} & Y \ar[r] \ar@{=}[d] & C_f\oplus Z \ar[d]^{p_{C_f}} \ar[r] & \Sigma X\oplus Z \ar[d]^{p_{\Sigma X}} \\
	X \ar[r]^f \ar@{=}[d] & Y \ar[r] \ar[d]^{i_Y} & C_f \ar[r] \ar[d]^{i_{C_f}} & \Sigma X \ar@{=}[d] \\
	X \ar[r]^{\left(\begin{array}{c}f\\ 0\end{array}\right)} & Y\oplus Z \ar[r] & C_f \oplus Z \ar[r] & \Sigma X }
\]
where $p_\ast$ and $i_\ast$ stand for the corresponding biproduct projections and inclusions.

Now since $C_f\oplus Z\in\D$, we know that both $\Loc\left(\left(\begin{array}{cc}f & 0\end{array}\right)\right)$ and $\Loc\left(\left(\begin{array}{c}f\\ 0\end{array}\right)\right)$ are isomorphisms. The two squares on the left therefore show that $\Loc(f)$ can be pre-composed with another morphism in $\C/\D$ such as to yield an isomorphism, and can also be post-composed with another morphism in $\C/\D$ such as to yield an isomorphism. By the 2-out-of-6 property of isomorphisms, it follows that $\Loc(f)$ is itself an isomorphism in $\C/\D$.

Conversely, suppose that $\Loc(f)$ is an isomorphism. Then we have $g,h\in\Mor(\C)$ as in Lemma~\ref{isocrit1}. Proposition~\ref{triplecomp} yields a diagram of the form
\[
\xymatrix{ & C_{fh} \ar[dr] \\
C_{h} \ar[ur] \ar[rr] & & C_{gfh} \ar[r] & C_{gf} \ar[r] & \Sigma C_h }
\]
in which $C_{gf}$ and $C_{fh}$ are isomorphic to objects in $\D$ by assumption. The claim now follows from Corollary~\ref{thickcrit}.
\end{proof}

We can now prove the main result of this section, stating that $\C/\D$ is indeed the Verdier localization of $\C$ with respect to $\W=\Iso(\D)$:

\begin{thm}
\label{loc}
$\C/\D$ is triangulated by taking a candidate triangle to be a triangle if and only if it is isomorphic to the image of a triangle under $\Loc^\triangle$. Furthermore, $\Loc^\triangle:\C\lra\C/\D$ is the universal triangulated functor mapping all objects of $\D$ to zero objects.
\end{thm}

\begin{proof}
The universal property of $\C/\D$ as a category of fractions shows that the suspension $\Sigma$ extends uniquely to a functor $\Sigma_{\C/\D}\: : \: \C/\D\lra \C/\D$ such that the diagram
\beq
\xymatrix{{\C}\ar[d]_{\Loc}\ar[r]^\Sigma & {\C}\ar[d]^{\Loc}\\
{\C/\D}\ar[r]_{\Sigma_{\C/\D}} & {\C/\D}}
\eeq
commutes. The same holds true for $\Sigma^{-1}$, and it follows that $\Sigma_{\C/\D}$ is invertible. 

For $\Loc$ to be triangulated, the image of any triangle in $\C$ has to be a triangle in $\C/\D$. Hence we can try to define a candidate triangle in $\C/\D$ to be a triangle if it is isomorphic (in $\C/\D$) to a triangle coming from $\C$. Then~\ref{T1} is trivial, while~\ref{T3} holds by definition. Concerning~\ref{T2}, any morphism in $\C/\D$ is of the form $\Loc(f)\Loc(w)^{-1}$ with $f$ a morphism of $\C$ and $w\in\Iso(\D)$, so let this be the morphism that we want to complete to a triangle. In $\C$, we can complete $f$ to a triangle $(f,g,h)$, and then there is an isomorphism of candidate triangles
\beq
\xymatrix@!C@C+33pt{{\bullet}\ar@{~>}[d]_{\Loc(w)}\ar[r]^{\Loc(f)} & {\bullet}\ar@{=}[d]\ar[r]^{\Loc(g)} & {\bullet}\ar@{=}[d]\ar[r]^{\Loc(h)} & {\bullet}\ar@{~>}[d]^{\Loc(\Sigma w)}\\
{\bullet}\ar[r]_{\Loc(f)\Loc(w)^{-1}} & {\bullet}\ar[r]_{\Loc(g)} & {\bullet}\ar[r]_{\Loc(\Sigma w)\Loc(h)} & {\bullet}}
\eeq
so that the second row is a triangle by definition.

As for~\ref{T4}, let $(f,g,h)$ be a triangle in $\C/\D$. By definition, we know that it is isomorphic to a triangle $(\Loc(f'),\Loc(g'),\Loc(h'))$ with $(f',g',h')$ being a triangle in $\C$, and that $(g',h',-\Sigma f')$ is also a triangle in $\C$. Hence in $\C/\D$, we have the diagram
\beq
\xymatrix@!C@C+33pt{ \bullet \ar[r]^{\Loc(f')} \ar@{~>}[d] & \bullet \ar[r]^{\Loc(g')} \ar@{~>}[d] & \bullet \ar[r]^{\Loc(h')} \ar@{~>}[d] & \bullet \ar[r]^{\Loc(-\Sigma f')} \ar@{~>}[d] & \bullet \ar@{~>}[d] \\
 \bullet \ar[r]^f & \bullet \ar[r]^g & \bullet \ar[r]^h & \bullet \ar[r]^{-\Sigma_{\C/\D} f} & \bullet }
\eeq
where all vertical morphisms are isomorphisms, with the last two being the $\Sigma_{\C/\D}$-suspensions of the first two. Hence $(g,h,-\Sigma_{\C/\D} f)$ is also a triangle in $\C/\D$.

It remains to verify~\ref{T5}, which is the most difficult part. It will turn out to be useful to have some auxiliary statements. First, we show that every composable pair of morphisms $(f,g)$ in $\C/\D$ can be lifted to a composable pair $(\hat{f},\hat{g})$ in $\C$ up to isomorphism as follows. Since $\C/\D$ is a category of fractions, we can write $f=\Loc(\hat{f})\Loc(v)^{-1}$ and $g=\Loc(w)^{-1}\Loc(\hat{g})$ with $v,w\in\Iso(\D)$, which means that we have a diagram
\[
	\xymatrix@!C@C+33pt{ \bullet \ar[r]^{\Loc(\hat{f})} \ar@{~>}[d]_{\Loc(v)} & \bullet \ar[r]^{\Loc(\hat{g})} \ar@{=}[d] & \bullet \ar@{~>}[d]^{\Loc(w)} \\
	\bullet \ar[r]_f & \bullet \ar[r]_g & \bullet }
\]
The relevant statement that we use in the following is that any composable pair in $\C/\D$ is isomorphic to a composable pair coming from $\C$ via conjugation with an isomorphism at each object; that these can be taken to be of the form $\Loc(v)$, $\id$ and $\Loc(w)$ will not be essential. Another auxiliary statement is the uniqueness of mapping cones in $\C/\D$, by which we mean the statement that if $(f,g,h)$ and $(f,g',h')$ are both triangles in $\C/\D$, then there is a diagram
\[
	\xymatrix@!C@C+33pt{ \bullet \ar[r]^f \ar@{=}[d] & \bullet \ar[r]^g \ar@{=}[d] & \bullet \ar[r]^h \ar@{-->}[d]^{\sim} & \bullet \ar@{=}[d] \\
	\bullet \ar[r]_f & \bullet \ar[r]_{g'} & \bullet \ar[r]_{h'} & \bullet}
\]
We interpret this as saying that the mapping cone triangle is unique up to isomorphism. To see that this indeed holds in $\C/\D$, we use the definition of triangle in terms of candidate triangles isomorphic to the $\Loc$-images of triangles, by which it is enough to prove the following statement: if $(\hat{f},\hat{g},\hat{h})$ and $(\hat{f}',\hat{g}',\hat{h}')$ are triangles in $\C$, and $\hat{f}$ is isomorphic to $\hat{f}'$ in $\Mor(\C)$, then there is an isomorphism of triangles in $\C/\D$
\[
	\xymatrix@!C@C+33pt{ \bullet \ar[r]^{\Loc(\hat{f})} \ar[d]^{\sim} & \bullet \ar[r]^{\Loc(\hat{g})} \ar[d]^{\sim} & \bullet \ar[r]^{\Loc(\hat{h})} \ar@{-->}[d]^{\sim} & \bullet \ar[d]^{\sim} \\
	\bullet \ar[r]_{\Loc(\hat{f}')} & \bullet \ar[r]_{\Loc(\hat{g}')} & \bullet \ar[r]_{\Loc(\hat{h}')} & \bullet}
\]
where the first two vertical isomorphisms define the given isomorphism between $\hat{f}$ and $\hat{f}'$ and the dashed isomorphism needs to be found. We start by decomposing the first vertical isomorphism (and its suspension at the very right) as a formal right fraction $\Loc(k)\Loc(v)^{-1}$ with $v\in\Iso(D)$, and the second vertical isomorphism as a formal left fraction $\Loc(w)^{-1}\Loc(l)$ with $w\in\Iso(\D)$. This turns the diagram into
\[
	\xymatrix@!C@C+33pt{ \bullet \ar[r]^{\Loc(\hat{f})} \ar@{<~}[d]^{\Loc(v)} & \bullet \ar[r]^{\Loc(\hat{g})} \ar[d]^{\Loc(l)} & \bullet \ar[r]^{\Loc(\hat{h})} \ar@{-->}[dd]^{\sim} & \bullet \ar@{<~}[d]^{\Loc(\Sigma v)} \\
	\bullet \ar[d]^{\Loc(k)} & \bullet \ar@{<~}[d]^{\Loc(w)} & & \bullet \ar[d]^{\Loc(\Sigma k)} \\
	\bullet \ar[r]_{\Loc(\hat{f}')} & \bullet \ar[r]_{\Loc(\hat{g}')} & \bullet \ar[r]_{\Loc(\hat{h}')} & \bullet}
\]
Now let us try to consider this diagram---without the yet-to-be-found dashed arrow---in $\mathcal{C}$ by ``erasing'' all applications of $\Loc$. There is clearly no guarantee that the resulting diagram will commute in $\C$, but by Lemma~\ref{fraceq} we can assume that the hexagon on the left commutes, i.e.~$w\hat{f}'k=l\hat{f}v$, without loss of generality: since $\Loc(w\hat{f}'k)=\Loc(l\hat{f}v)$, there exists $u\in\Iso(\D)$ such that $uw\hat{f}'k=ul\hat{f}v$, but then we can redefine $w$ and $l$ to be $uw$ and $ul$, respectively, thereby absorbing $u$. Thanks to this argument, we can strip off all applications of $\Loc$ and assume that this results in a commutative diagram in $\C$. Now we define the two intermediate horizontal triangles in
\[
	\xymatrix@!C@C+33pt@!R@R+15pt{ \bullet \ar[r]^{\hat{f}} \ar@{=}[d] & \bullet \ar[r]^{\hat{g}} \ar[d]^{l} & \bullet \ar[r]^{\hat{h}} \ar@{-->}[d]^r & \bullet \ar@{=}[d] \\
	\bullet \ar[r]^{l\hat{f}} \ar@{<~}[d]^{v} & \bullet \ar[r] \ar@{<~}[d]^{w} & \bullet \ar[r] \ar@{<--}[d]^s & \bullet \ar@{<~}[d]^{\Sigma v} \\
	\bullet \ar[r]^{\hat{f}'k} \ar[d]^{k}  & \bullet \ar[r] \ar@{=}[d] & \bullet \ar[r] \ar@{-->}[d]^t & \bullet \ar[d]^{\Sigma k} \\
	\bullet \ar[r]_{\hat{f}'} & \bullet \ar[r]_{\hat{g}'} & \bullet \ar[r]_{\hat{h}'} & \bullet}
\]
to be triangles generated by the newly introduced morphisms $l\hat{f}$ and $\hat{f}'k$, which we can picture as diagonals in the previous diagram. Now the morphisms $r$, $s$ and $t$ should be chose to be certain filling morphisms in $\C$ as follows. First, an application of the $3\!\times\! 3$-lemma (Lemma~\ref{3x3lemma}) to the middle square on the left shows that one can choose $s\in\Iso(\D)$, since $\D$ is closed under taking mapping cones. Similarly, by Proposition~\ref{isocrit2} we know that $k,l\in\Iso(\overline{\D})$, and hence one can also choose $r,t\in\Iso(\overline{D})$, which again implies that $\Loc(r)$ and $\Loc(t)$ are isomorphisms in $\C/\D$. Taken all together, this finishes the proof of our second auxiliary statement on the uniqueness of mapping cones in $\C/\D$.

Using these two auxiliary statements, we can now prove that~\ref{T5} holds in $\C/\D$. By the uniqueness of mapping cones, it is enough to start with any composable pair of morphisms in $\C/\D$, complete them to triangles in a unique-up-to-isomorphism way, and then show that the conclusion of~\ref{T5} is indeed true. Moreover, again by uniqueness of mapping cones, it is sufficient to show this only for one representative of every isomorphism class of composable pairs of morphisms. By the first auxiliary statement, it is therefore sufficient to assume that the composable pair is of the form $\Loc(g)\circ \Loc(f)$ for $f,g\in\Mor(\C)$. But in this case, the claim follows from~\ref{T5} in $\C$.

Finally, we show that with this triangulation, $\C/\D$ has the desired universal property. Clearly, $\Loc\: :\: \C\lra \C/\D$ maps all objects in $\D$ to zero objects by Corollary~\ref{iso}. Then any other triangulated functor mapping the objects of $\D$ to zero also maps all $w\in\Iso(\D)$ to isomorphisms. Hence as a plain functor, it uniquely factors over $\Loc$, since $\Loc$ was constructed as a localization~\cite[Thm.~3.9]{Fri1}. By definition of the suspension on $\C/\D$, this factored functor commutes with the suspensions; and by definition of the triangulation on $\C/\D$ and by~\ref{T3} in the target category, it necessarily preserves triangles.
\end{proof}

\begin{cor}
\label{localhomfunctor}
If $H:\C\lra\Ab$ is a homological functor vanishing on all of $\D$, then $H$ factors uniquely over $\Loc$ and thus induces a homological functor $\C/\D\lra\Ab$.
\end{cor}

\begin{proof}
By long exact sequences, $H$ maps all $w\in\Iso(\D)$ to isomorphisms of abelian groups. But then the assertion follows from the universal property of the category of fractions $\C/\D$.
\end{proof}

\begin{prop}
\label{verdierkernel}
It holds that $\ker(\Loc:\C\lra\C/\D)=\overline{\D}$.
\end{prop}

\begin{proof}
Recall that an object in an additive category is a zero object if and only if its zero endomorphism is an isomorphism. Thus we need to show that $\Loc(0_X)$ is an isomorphism if and only if $X\in\overline{D}$. By Proposition~\ref{isocrit2}, the former happens if and only if the mapping cone of $0_X:X\lra X$ lies in $\D$. Since this mapping cone is $X\oplus\Sigma X$, the claim follows from Corollary~\ref{thickness}.
\end{proof}

We can now prove the existence of the Verdier localization $\C[\W^{-1}]^\triangle$ with respect to any localizing class $\W$ and characterize its kernel:

\begin{thm}
\label{mainthm}
Let $\W$ be any class of morphisms in $\C$, and let $\D_\W$ be the thick triangulated subcategory generated by the class of mapping cones $\{C_w,\,w\in\W\}$. Then $\C[\W^{-1}]^\triangle=\C/\D_\W$, and the kernel of $\Loc^\triangle$ is exactly $\D_\W$.
\end{thm}

\begin{proof}
Any triangulated functor which takes all morphisms in $\W$ to isomorphisms also takes all mapping cones $C_w$ to zero objects, and therefore takes all of $\D_\W$ to zero objects. Hence such a functor uniquely factors over $\C/\D_\W$ by Theorem~\ref{loc}. On the other hand, $\C/\D_\W$ itself already has the property that all morphisms in $\W$ become invertible.
\end{proof}

An important concern in many applications is that a Verdier localization $\C/\D$ or $\C[\W^{-1}]^\triangle$ need not be locally small if $\C$ is a large category: as a category of fractions, a hom-set in such a Verdier localization is a quotient of a proper class\footnote{By a ``proper class'' we really mean a set in a higher Grothendieck universe.} containing all the corresponding roofs by an equivalence relation. This ought to be kept in mind when applying Theorems~\ref{loc} and~\ref{mainthm}. It is known that under certain conditions which frequently hold in practice, one can construct the Verdier localization as a locally small category; see e.g.~Theorem~4.4.9 and Remark 9.1.17 in~\cite{Nee}.

\section{Examples}
\label{se:ex}

We end these notes with a collection of example situations and contexts in which triangulated categories arise in various areas of mathematics.

Broadly speaking, the two most important constructions of triangulated categories are the following two: first the homotopy category of a stable model category is triangulated (Example~\ref{model}); second, the stable category of a Frobenius category is triangulated (Example~\ref{frob}). Triangulated categories which arise from stable model categories are called~\emph{topological}, while triangulated categories that arise from Frobenius categories are called~\emph{algebraic}~\cite{Schwede,Schwede2}. There are triangulated categories that are neither topological nor algebraic~\cite{MSS}, but as far as we know, all ``naturally arising'' triangulated categories are algebraic or topological, while many actually are both. 

\subsection{Semisimple abelian categories}
\label{semisimpleab}
These are toy examples with little relevance beyond being relatively elementary examples of triangulated categories in which explicit computations can be performed easily.

Let $\A$ be a semisimple abelian category, i.e.~an abelian category in which every short exact sequence
\[
\xymatrix{ 0 \ar[r] & A \ar[r] & B \ar[r] & C \ar[r] & 0 }
\]
satisfies the following three equivalence conditions:
\begin{enumerate}
\item The epimorphism $B\lra C$ has a right inverse (section);
\item The monomorphism $A\lra B$ has a left inverse (retraction);
\item There is an isomorphism $B\cong A\oplus C$ such that the original short exact sequence is isomorphic to a biproduct short exact sequence
\[
\xymatrix{ 0 \ar[r] & A \ar[r] & A\oplus C \ar[r] & C \ar[r] & 0 }
\]
\end{enumerate}
In this situation, take $\Sigma:=\id_\A$, which is trivially an automorphism. Then define a candidate triangle
\[
\xymatrix{ X \ar[r]^f & Y \ar[r]^g & Z \ar[r]^h & X }
\]
to be a triangle if and only if it is exact at all three objects, where exactness at $X$ means that $\ker(f) = \coker(h)$. With this definition,~\ref{T1}--\ref{T4} are all immediate, so that it only remains to prove~\ref{T5}. For this, we need to take a closer look at distinguished triangles, and this is where the semisimplicity of $\A$ will become relevant. So for any triangle as above, we have $\ker(h)=\im(g)$ by exactness at $Z$, but also $\im(g)\cong Y/\ker(g)$ by the isomorphism theorem, and finally $Y/\ker(g)\cong \coker(f)$ by exactness at $Y$. Taken together, this yields $\ker(h)\cong \coker(f)$. Dually, we can likewise conclude $\coker(g) \cong \ker(f)$,

Semisimplicity tells us that the short exact sequence
\[
\xymatrix{ 0 \ar[r] & \ker(h) \ar[r] & Z \ar[r] & \coker(g) \ar[r] & 0 }
\]
splits, and hence $Z\cong \ker(h)\oplus \coker(g)$. By the previous observations, this also shows that $Z\cong \ker(f) \oplus \coker(f)$. In particular, the original triangle that we started with is isomorphic to the triangle
\beq
\xymatrix{X\ar[r]^f & Y \ar[r] & \ker(f)\oplus\coker(f)\ar[r] & X}
\eeq
in which the second morphism comes from the cokernel projection $Y\lra \coker(f)$ and the third from the kernel inclusion $\ker(f)\lra X$. In conclusion, we have identified $\ker(f)\oplus\coker(f)$ as a mapping cone of $f$. Conversely, it is easy to show that every triangle of this new form is indeed exact at all three objects, so that we can also regard the new form of the triangle as a characterization of triangles.

We now sketch the verification of~\ref{T5}. For any morphism $f :  X\lra Y$, we write $i_f  : \ker(f)\lra X$ for the inclusion of its kernel into the domain, and $p_f  : Y\lra \coker(f)$ for the projection from the codomain onto its cokernel. Then~\ref{T5} takes on the form
\[
\xymatrix@!@-4pc{ X \ar@(ru,lu)[rr]^{gf}\ar[rd]_f & & Z \ar@(ru,lu)[rr]^{\left(\begin{array}{c} 0\\ p_g\end{array}\right)} \ar[rd]^{\left(\begin{array}{c} 0\\ p_{gf}\end{array}\right)} & & \ker(g)\oplus\coker(g) \ar@(ru,lu)@{-->}[rr]^{\left(\begin{array}{cc} 0 & 0\\ p_f i_g & 0 \end{array}\right)} \ar[rd]^{\left(\begin{array}{cc}i_g & 0\end{array}\right)} & & \ker(f)\oplus\coker(f) \\
& Y \ar[ur]^g \ar[rd]_{\left(\begin{array}{c} 0\\ p_f\end{array}\right)} & & \ker(gf)\oplus\coker(gf) \ar@{-->}[ru]_{k'} \ar[rd]_{\left(\begin{array}{cc} i_{gf} & 0\end{array}\right)} & & Y \ar[ru]_{\left(\begin{array}{c} 0 \\ p_f \end{array}\right)} &\\
& & \ker(f)\oplus\coker(f) \ar@(rd,ld)[rr]_{\left(\begin{array}{cc} i_f & 0 \end{array}\right)} \ar@{-->}[ru]^{k} & & X \ar[ru]_{\Sigma f} & &}
\]
The problem here is to find suitable $k$ and $k'$. Using the universal properties of kernel and cokernel, one obtains morphisms
\[
\xymatrix@R-6pt{ \ker(f) \ar[r]^{j_1} & \ker(gf)   \ar[r]^{j_2} &   \ker(g) \\
 	 \coker(f) \ar[r]^{q_1} & \coker(gf) \ar[r]^{q_2} & \coker(gf) }
\]
and then one can take $k$ and $k'$ to be given by
\[
k:=\left(\begin{array}{cc} j_1 & 0 \\ 0 & q_1 \end{array}\right),\qquad k':=\left(\begin{array}{cc} j_2 & 0 \\ 0 & q_2 \end{array}\right).
\]
It is straightforward but laborious to check that this makes the above diagram commute and turns the three dashed arrows into another triangle.

Finally, it is worth noting that every triangle in $\A$ can actually be written as a direct sum of triangles of the form
\[
\xymatrix@R-1pc{ X \ar@{=}[r] & X \ar[r] & 0 \ar[r] & X \\
		 0 \ar[r] & X \ar@{=}[r] & X \ar[r] & 0 \\
		 X \ar[r] & 0 \ar[r] & X \ar@{=}[r] & X }
\]
This follows from the above intermediate form by decomposing $X\cong \ker(f)\oplus\im(f)$ and $Y\cong \im(f)\oplus \coker(f)$ using semisimplicity. This can be interpreted as stating that $\A$ does not contain any non-trivial triangles, and in this sense $\A$ is a toy example of a triangulated category. Another indication for the latter statement is that \emph{every} additive functor $\xymatrix{ \A \ar[r] & \Ab }$ is homological, again by semisimplicity together with the fact that every additive functor preserves biproducts.

\subsection{Derived categories}
\label{derived}

One triangulated category which can be constructed out of any abelian category $\A$ is its homotopy category of chain complexes $\K(\A)$~\cite[Prop.~10.2.4]{Weibel}. However, of much bigger significance is the \emph{derived category} $\D(\A)$, as introduced by Verdier~\cite{Ver}. The goal here is to find a triangulated category on which ``derived functors'' together with their long exact sequences can easily be defined as homological functors acting on Puppe sequences. Derived functors are ubiquitous in homological algebra and its fields of application, spanning a wide range of situations ranging from $\mathrm{Ext}$ and $\mathrm{Tor}$ on categories of modules~\cite[Ch.~2]{Weibel} via group (co-)homology and Lie algebra (co-)homology~\cite[Ch.~6/7]{Weibel} to sheaf cohomology~\cite[Sec.~III.8]{GM}.

The derived category $\D(\A)$ can be constructed as a Verdier localization of the homotopy category of chain complexes, where the localizing subcategory is given by all those chain complexes which have vanishing homology in all degrees. On the level of morphisms, this means that the localizing class $\W$ consists of the \emph{quasi-isomorphisms}, i.e.~those (homotopy classes) of maps which are degreewise isomorphisms on homology. One important reason for considering quasi-isomorphisms is that one would like a short exact sequence in $\A$
\[
\xymatrix{ 0 \ar[r] & A \ar[r] & B \ar[r] & C \ar[r] & 0 }
\]
to correspond to a triangle in $\D(\A)$, so that it determines a long exact sequence under any homological functor on $\D(\A)$. Now the object $A\in\A$ corresponds to the chain complex with $A$ in degree $0$ and zero objects everywhere else. Using the fact that the mapping cones in $\K(\A)$ are given by the usual mapping cones of chain complexes~\cite[Sec.~1.5]{Weibel}, we obtain that the inclusion $A\lra B$ has a mapping cone triangle in $\K(\A)$ given by the diagram of chain complexes
\[
\xymatrix{ \ldots \ar[r] & 0 \ar[r] \ar[d] & 0 \ar[r] \ar[d] & A \ar[r] \ar[d] & 0 \ar[r] \ar[d] & \ldots \\
	   \ldots \ar[r] & 0 \ar[r] \ar[d] & 0 \ar[r] \ar[d] & B \ar[r] \ar[d] & 0 \ar[r] \ar[d] & \ldots \\
	   \ldots \ar[r] & 0 \ar[r] \ar[d] & A \ar[r] \ar[d] & B \ar[r] \ar[d] & 0 \ar[r] \ar[d] & \ldots \\
	   \ldots \ar[r] & 0 \ar[r]        & A \ar[r]        & 0 \ar[r]        & 0 \ar[r]        & \ldots }
\]
where the triangle extends vertically and all morphisms are the obvious ones. (The suspension functor in $\K(\A)$ is simply given by shifting a chain complex by one degree, as indicated in the diagram with the first and last row.) In general, the chain complex in the third row is not homotopy equivalent to the one having only $C$ in degree $0$. However, the original short exact sequence defines a quasi-isomorphism
\[
\xymatrix{ \ldots \ar[r] & 0 \ar[r] \ar[d] & A \ar[r] \ar[d] & B \ar[r] \ar[d] & 0 \ar[r] \ar[d] & \ldots \\
	   \ldots \ar[r] & 0 \ar[r]        & 0 \ar[r]        & C \ar[r]        & 0 \ar[r]        & \ldots }
\]
which identifies $C$ with the mapping cone of $A\lra B$ in $\K(\A)$. In this way, any homological functor on the derived category $\D(\A)$ yields long exact sequences from short exact sequences in $\A$.

There are important triangulated subcategories of $\D(\A)$ on which one can construct many homological functors relatively explicitly. In particular, there is a full subcategory $\D^+(\A)$ containing all those chain complexes which are \emph{bounded below}, i.e.~which consist of zero objects below a certain degree. Since this class of chain complexes is closed under taking mapping cones, $\D^+(\A)$ is a triangulated subcategory of $\D(\A)$. If $\A$ has enough injectives, then $\D^+(\A)$ can also be described in a different way as the homotopy category of bounded below chain complexes of injective objects~\cite[Thm.~III.5.21]{GM}. This is why \emph{derived functors} like $\mathrm{Ext}$ and $\mathrm{Tor}$, group (co)homology or Hochschild (co)homology~\cite{Weibel} are traditionally constructed in terms of injective (or, dually, projective) resolutions. Similar statements apply to the 

\subsection{Stable homotopy categories of model categories}
\label{model}

Model categories provide an abstract framework for \emph{homotopical algebra}~\cite{Quillen,Hov}. This means that they capture the essential structures of ordinary homotopy theory by providing abstract and general definitions and theorems applying in many other situations, such as for example the $\mathbf{A}^1$-homotopy theory of schemes in algebraic geometry~\cite{MV}. In particular, in any model category one can talk about homotopy equivalences (typically called \emph{weak equivalences}), suspensions and loop space objects.

The most important other category associated to a model category is its \emph{homotopy category}. One way to obtain it is by localizing the model category at all weak equivalences, but one can also obtain the homotopy category by other means~\cite[Thm.~1.2.10]{Hov}. This homotopy category is close to being a triangulated category in the sense that it has some similar properties, although the suspension functor is not invertible in general; if the model category is pointed (i.e.~has a zero object), then the homotopy category is pre-triangulated in the sense of Hovey~\cite[Ch.~6]{Hov}.

Moreover, if a model category is \emph{stable} in the sense that the suspension functor is an equivalence, then its homotopy category is indeed triangulated~\cite[Ch.~7]{Hov}. For example, the category of bounded below chain complexes in an abelian category with enough injectives can be turned into a stable model category in such a way that the homotopy category turns out to be the derived category $\D^+(\A)$~\cite[p.~1.16]{Quillen}; in this sense, the present example comprises the previous one. But also for pointed model categories that are not necessarily stable, it is shown in~\cite{Del} that one can take the homotopy category, stabilize with respect to the suspension, and one obtains a triangulated category. Here, stabilizing with respect to the suspension means that one defines a new category in which the new morphisms $X\lra Y$ are be given by morphisms between iterated suspensions $\Sigma^n X\lra \Sigma^n Y$ in the original category in the sense of taking the (co-)limit as $n\to\infty$.

\subsection{Stable Homotopy Theory and Spectra}
\label{spectra}

These are actually two examples which are special cases of the previous one. First, one can start with the category of CW-complexes, consider its homotopy category, and stabilize as in the previous example. This gives a triangulated category known as the \emph{Spanier-Whitehead category} or the category of \emph{finite spectra}~\cite{Strickland}; see also~\cite{Margolis}, where the term ``Spanier-Whitehead category'' is used for the analogous category without the finiteness assumption. However, with either definition, the Spanier-Whitehead category does not yet have the desired completeness properties; in some sense, it is ``too small''~\cite[p.~vii]{Margolis}. Hence one usually works with the more intricate \emph{category of spectra} instead. In the algebraic topology literature, the term ``stable homotopy category'' always refers to the stable homotopy category of spectra. It has the appealing feature of \emph{Brown representability} alluded to in Remark~\ref{brownrep}, so that every cohomology theory can be represented by a spectrum.

\subsection{Stable categories of Frobenius categories}
\label{frob}

First, a \emph{Quillen exact category} $\Q$ is a full additive subcategory of an abelian category which is, in addition, closed under extensions, i.e.~if $A\in\Q$ and $B\in\Q$ and there is a short exact sequence
\[
\xymatrix{ 0 \ar[r] & A \ar[r] & B \ar[r] & C \ar[r] & 0 }
\]
then also $B\in\Q$. Quillen exact categories can also be characterized in terms of purely intrinsic axioms not referring to any embedding category~\cite{Bühler}. A \emph{Frobenius category} is then a Quillen exact category which has enough injectives and projectives, and such that the class of injective objects coincides with the class of projective objects. Some abelian categories themselves are Frobenius categories, such as the category of representations of a finite group; if the characteristic of the ground field divides the order of the group, then this category is typically not semisimple, and this makes the example non-trivial. More concretely, a non-trivial example of a Frobenius category is the category of modules over the ring $\Z/4\Z$. More generally, one can consider module categories of self-injective algebras, i.e.~algebras which are injective as a module over themselves.

One can \emph{stabilize} a Frobenius category $\Q$ by regarding two morphisms as equivalent if and only if their difference factors over an injective (or equivalently projective) object, and then forming the quotient category with respect to this equivalence relation. This \emph{stable category} is triangulated in a canonical way~\cite[Ch.~2]{Happel}: every $A\in\Q$ embeds into an injective object $I(A)\in\Q$ resulting in a short exact sequence
\[
\xymatrix{ 0 \ar[r] & A \ar[r] & I(A) \ar[r] & I(A)/A \ar[r] & 0 }
\]
and the suspension is then defined by $\Sigma A:= I(A)/A$. Any other short exact sequence
\[
\xymatrix{ 0 \ar[r] & A \ar[r] & B \ar[r] & C \ar[r] & 0 }
\]
gives rise to a triangle by considering the induced diagram
\[
\xymatrix{ 0 \ar[r] & A \ar[r] \ar@{=}[d] & B \ar[r] \ar[d] & C \ar[r] \ar[d] & 0 \\
	   0 \ar[r] & A \ar[r]            & I(A) \ar[r]     & \Sigma A \ar[r] & 0 }
\]
as a diagram in the stable category and restricting to $A\lra B\lra C\lra \Sigma A$.

The stable category of the category of modules over a self-injective algebra can also be characterized as a Verdier localization of the bounded derived category of the algebra~\cite[Thm.~2.1]{Rickard}.

\subsection{$K$-theory of $C^*$-algebras}

In algebraic topology, $K$-theory~\cite{HatcherK} is one of the most well-known cohomology theories. Moreover, $K$-theory can be defined more generally for $C^*$-algebras~\cite{RLL}, which play the role of ``noncommutative spaces''. In fact, $K$-theory for $C^*$-algebras, and hence for suitably nice topological spaces in particular, comes in a more general version known as $KK$-theory; see~\cite{Higson,Blackadar} for an introduction. This generalization is \emph{bivariant} in the sense that it is a functor $KK(\cdot,\cdot)$ taking two $C^*$-algebras as arguments, contravariantly in the first and covariantly in the second, and returning an abelian group. For three $C^*$-algebras $A,B,C$, there is a bilinear map
\[
\xymatrix{ KK(A,B) \:\times\: KK(B,C) \ar[r] & KK(A,C) }
\]
satisfying the appropriate naturality conditions. By virtue of this, $KK(\cdot,\cdot)$ is the hom-functor in a category also denoted $KK$, and in fact this category is triangulated~\cite{Meyer}. The suspension functor on $KK$ assigns to every $C^*$-algebra $A$ the $C^*$-algebra
\[
\Sigma A := \{ \: f : [0,1] \stackrel{\mathrm{cont.}}{\lra} A \:|\: f(0)=f(1)=0 \:\}.
\]
Under Gelfand duality---in the sense of the contravariant equivalence between the category of commutative $C^*$-algebras and the category of pointed compact Hausdorff spaces---this $\Sigma$ specializes to the usual reduced suspension of pointed topological spaces.

Despite its wide range of applications, $KK$-theory has a deficit: a short exact sequence of $C^*$-algebras
\[
\xymatrix{ 0 \ar[r] & A \ar[r] & B \ar[r] & C \ar[r] & 0 }
\]
does not in general yield a distinguished triangle in $KK$~\cite{Skandalis}. This is improved upon by $E$-theory, which is another triangulated category with $C^*$-algebras as objects which can be understood as the universal Verdier localization of $KK$ which turns all short exact sequences into triangles~\cite{Hig}.

It is known that both triangulated categories $KK$ and $E$ can be obtained as homotopy categories of categories of fibrant objects~\cite{Uuye}; here, the notion of ``category of fibrant objects'' generalizes that of model category (Example~\ref{model}). One can also obtain $KK$ as the homotopy category of a genuine model category if one generalizes from $C^*$-algebras to so-called $l.m.c.$-$C^*$-algebras~\cite{JJ}.

\subsection{Motives}
\label{motives}

In algebraic geometry, \emph{motives} are partly hypothetical objects forming a category which supposedly ``linearizes'' the category of algebraic varieties. More precisely, the category of motives should be the universal abelian category on which any algebro-geometric cohomology theory naturally operates. However, constructing such a category has been an open problem since the 1960's~\cite{Grothendieck}. Only for the case of projective varieties has a candidate category been constructed, and this is known as the \emph{category of pure motives}. At present, there are two variants of the category of pure motives: one using rational equivalence of algebraic cycles in its definition, and the other using numerical equivalence of algebraic cycles~\cite{Milne}. The former gives a category through which every cohomology theory factors, but this category is not abelian; the latter gives an abelian category, but proving that any Weil cohomology factors through it is one of the difficult open ``standard conjectures''.

In the general case of not necessarily projective varieties, one deals with \emph{mixed motives}, a significantly more difficult case. This is where triangulated categories come in: so far, only the \emph{derived} category of the desired abelian category of mixed motives has been constructed as a triangulated category~\cite{VSF}, and the hope is that one can recover a category of mixed motives from this triangulated category.

Motives also have applications beyond traditional algebraic geometry, for example to quantum field theory~\cite{Marcolli}.

\bigskip

This list of examples of triangulated categories is necessarily incomplete and biased by the background and interests of the author. For others, see e.g.~singularity categories in commutative algebra and algebraic geometry~\cite{Orlov}, (cohomological categories of) Fukaya categories of symplectic manifolds in symplectic geometry~\cite{Auroux}, or cluster categories in representation theory~\cite[3.2]{Keller}.

\section{Where to go from here?}

These notes may have given the reader some ideas of what triangulated categories are about, what technical tools there are available for working with them, and in what fields of mathematics they arise. What else is there to say? In which further directions has the general theory been developed?

Here, we would like to give a few more pointers to the literature concerning the general theory. Due to the breadth of applications of triangulated categories, this again necessarily represents a biased sample. First, as already mentioned in Remark~\ref{brownrep}, there is a general version of the Brown representability theorem (Example~\ref{spectra}); see e.g.~Chapter 8 of~\cite{Nee}, and also the related topic of Bousfield localization in Chapter 9, which is concerned with the existence of adjoints to a Verdier localization functor. Other notions relating to triangulated subcategories and localizations are those of \emph{left and right orthogonal subcategories} and \emph{recollements}. For these topics and more on localization in general, see~\cite{Krause}. Using these notions, one can define a \emph{semiorthogonal decomposition} of a triangulated category to consist of a finite sequence of triangulated subcategories from which one can ``build up'' the given triangulated category in a suitable sense; see e.g.~\cite{Orlov}. A prominent topic in modular representation theory is that of \emph{Serre functors} and \emph{Serre duality} of triangulated categories, leading up to the \emph{Calabi-Yau} property that certain triangulated categories enjoy~\cite{Keller}. The so-called \emph{t-structures} and \emph{hearts} figure prominently as additional data that one can use to reconstruct an abelian category from its derived category~\cite[IV.4]{GM}, and they are relevant to Example~\ref{motives} in that one can try to construct an abelian category of mixed motives in terms of a heart of the triangulated category of mixed motives. In many situation in practice, a triangulated category comes equipped with a monoidal structure. If the monoidal structure is suitably compatible with the triangulation, then one speaks of a \emph{tensor triangulated category}, for which a rich theory exists~\cite{Bal2}.

Finally, let us note that triangulated categories have been generalized to \emph{$n$-angulated categories}~\cite{nang}.

\section{Refining triangulated categories}
\label{refine}

Finally, it should be mentioned that there is widely known evidence for the hypothesis that triangulated categories are not the most adequate notion for describing the phenomena that they have been introduced for. One problem is that the formation of mapping cones is only weakly functorial (Proposition~\ref{fillingmorphism}); in fact, by a result of Verdier~\cite{Stev}, under reasonable conditions it is impossible to choose mapping cones in a functorial way. Related issues arise with respect to the formation of homotopy limits and colimits, which is difficult to achieve due to the fact that a triangulated category does not ``remember'' any homotopical information. There are several notions that refine triangulated categories and overcome these problems to different degrees by retaining more homotopical information than is present in a triangulated category; the most prominent of these notions are stable $(\infty,1)$-categories~\cite{Lurie}, stable derivators~\cite{Heller,Groth}, $A_\infty$-categories~\cite{BLM}, or pretriangulated dg-categories~\cite{BK}. Any instance of any of these structures can be turned into a triangulated category in a canonical way, and doing so in the case of the standard examples recovers precisely the usual triangulated categories of Section~\ref{se:ex}. There is also the notion of stable model category (Example~\ref{model}), and it has been hypothesized that ``every triangulated category that arises in nature is the homotopy category of a stable model category''~\cite[Ch.~7]{Hov}. However, since a stable model category can be thought of as a presentation of a stable $(\infty,1)$-category~\cite{Lurie}, this is subsumed by the first approach.

In this sense, the theory of triangulated categories should be studied with one caveat in mind: triangulated categories have many different uses, but it may be the case that all of these can better be captured by other categorical structures.

\bibliographystyle{intro_triang_cats}
\bibliography{intro_triang_cats}

\end{document}